\documentclass[12pt,reqno]{amsart}
\usepackage{amssymb,amsmath,amsthm,amscd, eucal}
\usepackage{color} 
\usepackage{enumerate}
\usepackage{bbm}
\usepackage[a4paper]{geometry}
\usepackage{tikz}
\usetikzlibrary{arrows}
\usepackage{amscd}

\DeclareMathOperator{\sign}{sign} 
\DeclareMathOperator{\Ran}{Ran}

\DeclareMathOperator{\Ker}{Ker}

\DeclareMathOperator{\clos}{clos}

\renewcommand\Im{\text{\rm Im}\,}
\renewcommand\Re{\text{\rm Re}\,}

\newcommand{\q}{\quad}

\newcommand{\ov}{\overline}
\newcommand{\z}{\zeta}

\newcommand{\e}{\eqref}
 \newcommand{\1}{\mathbbm{1}}
 
\numberwithin{equation}{section}

\theoremstyle{plain}
\newtheorem{theorem}{\bf Theorem}[section]
\newtheorem{lemma}[theorem]{\bf Lemma}
\newtheorem{proposition}[theorem]{\bf Proposition}

\newtheorem{corollary}[theorem]{\bf Corollary}

\theoremstyle{definition}
\newtheorem{definition}[theorem]{\bf Definition}

\newtheorem{cond}[theorem]{\bf Condition}

\theoremstyle{remark}
\newtheorem*{remark*}{\bf Remark}
\newtheorem{remark}[theorem]{\bf Remark}

\begin{document}

\title{On spectral analysis of self-adjoint Toeplitz operators}

\author{Alexander Sobolev}
\address{Department of Mathematics, 
Gower Street, London WC1E 6BT, U.K.}
\email{a.sobolev@ucl.ac.uk}

\author{Dmitri Yafaev}
\address{Univ  Rennes, CNRS, IRMAR-UMR 6625, F-35000 
    Rennes, France and SPGU, Univ. Nab. 7/9, Saint Petersburg, 199034 Russia}
\email{yafaev@univ-rennes1.fr}


\thanks{Our collaboration has become possible through the hospitality and financial support 
of the Departments of Mathematics  of University College London and of the University of Rennes 1.
The LMS grant is gratefully acknowledged. 
The authors were also supported 
by EPSRC grant EP/J016829/1 (A.S.) and RFBR grant No. 17-01-00668 A (D. Y.).} 

\begin{abstract}
The paper pursues three objectives. 
Firstly, 
we provide an expanded version of spectral analysis 
of self-adjoint Toeplitz operators, initially built by 
M. Rosenblum in the 1960's. We offer some improvements to Rosenblum's approach: 
for instance, our proof of the absolute continuity, relying 
on a weak version of the limiting absorption principle,
is more direct. 

Secondly, we study in detail Toeplitz operators with finite spectral multiplicity. 
In particular, we introduce generalized eigenfunctions and investigate their properties. 

Thirdly, we develop a more detailed   spectral analysis  
for piecewise continuous symbols. This is necessary for construction of
scattering theory for Toeplitz operators with such symbols. 
\end{abstract}



\maketitle

\section{Introduction} 
It is known  \cite{HaWi} that the spectrum $\sigma (T)$ of a 
self-adjoint Toeplitz operator $T =T(\omega)$  
with a real semi-bounded symbol $\omega(\z)$ on the unit cirlce $\mathbb T$, coincides with the interval 
\begin {equation}
\sigma (T)= [\gamma_1 , \gamma_2] \q{\rm where}\q
\gamma_1 = \textup{ess-inf}\;\omega, \; \gamma_2 = \textup{ess-sup}\;\omega,
\label{eq:sig}\end{equation}
and it is absolutely continuous  \cite{Ros1} unless $\omega$ is a constant. 
We note also the papers \cite{ Ismag,Ros3} where the 
multiplicity of the spectrum $\sigma (T)$  was found, 
and \cite{Ros2, Ros3} where a spectral representation of $T$ was constructed.
A presentation of spectral theory for self-adjoint Toeplitz operators can be also found 
in \cite[Chapter 3, Examples and Addenda]{RosRov}. 
For analysis of general Toeplitz operators see, e.g., 
the books \cite{Boet, NK,  NKN, Pe}.

Our ultimate objective is to construct  (in the forthcoming paper \cite{SY}) scattering theory for 
Toeplitz operators $T(\omega)$  with piecewise continuous symbols $\omega$.
The case of   symbols $\omega$ with jump discontinuities seems to be particularly 
interesting because for such symbols 
scattering theory becomes multichannel.   
In the current paper we present spectral analysis of self-adjoint Toeplitz operators, adapted for this application.

Our starting point is M. Rosenblum's papers \cite{Ros1, Ros2, Ros3}. 
Since Rosenblum's presentation is rather condensed and sometimes sketchy, we believe that 
interested specialists would benefit from a more detailed exposition of the relevant results. 
Thus the first
aim of the present paper is to a large extent methodological --  
to provide an expanded and detailed spectral analysis of general self-adjoint 
Toeplitz operators, see Sections 
\ref{toeplitz_forms:sect}, \ref{genT:sect}. Although we mostly follow Rosenblum's construction, 
our proof of the absolute continuity of $T$ in Theorem \ref{ac:thm} 
is more direct compared to \cite{Ros1} or \cite{RosRov}. It relies 
on a weak version of the limiting absorption principle, which is 
established via a straightforward application of Jensen's inequality.    

The second aim of the paper is to study Toeplitz operators with finite spectral multiplicity. 
In this case general results of Sect. \ref{genT:sect} can be made more explicit. 
In particular, we introduce generalized eigenfunctions 
and study their properties. The eigenfunctions are used to 
construct a unitary operator that diagonalizes $T$, 
i.e. realizes a spectral representation of $T$, 
see Theorem \ref{specuni:thm}. 

The third aim of the paper is to derive 
a convenient formula for 
the (finite) spectral multiplicity of 
Toeplitz operators with piecewise continuous symbols. This result is crucial 
for our construction of scattering theory in \cite{SY}. 

To summarize, this paper supplements Rosenblum's articles 
\cite{Ros1, Ros2, Ros3}, and it is a prerequisite for \cite{SY}.

The more detailed plan of the paper is as follows: 
Sect. \ref{toeplitz_forms:sect} 
contains basic definitions and a convenient formula for the bilinear form of the resolvent 
$(T-z)^{-1}$. 
In Sect.~\ref{genT:sect} we prove the absolute continuity of $T$ and provide 
 a formula for the spectral family of $T$. 
From Sect.~\ref{diag:sect} onwards, we impose Condition \ref{mult:cond} which ensures that 
the spectrum of $T$ on a fixed interval $\Lambda\subset\mathbb R$ is of finite multiplicity.
In this case  general results of Sect. \ref{genT:sect} can be made more explicit. 
In particular, we introduce  generalized 
eigenfunctions (or the continuous spectrum eigenfunctions) 
of the Toeplitz operator and produce a formula 
(see Theorem \ref{multiple:thm}) 
for its spectral family in terms of these functions. 
 The latter are used to construct a unitary operator that diagonalizes $T$, see Theorem \ref{specuni:thm}.  
Properties of eigenfunctions of Toeplitz operators are studied in Sect.~\ref{eigen:sect} 
where we establish a link with the Riemann-Hilbert problem, 
see \cite{Gahov} for information 
on Riemann-Hilbert problems. 
Here we also discuss two examples of symbols for which the eigenfunctions can be found explicitly. 

The final 
Sect.~\ref{piecewise:sect} is devoted to Toeplitz operators with piecewise continuous symbols.   
In this case Condition 
\ref{mult:cond}, which guarantees that the multiplicity is finite, 
is satisfied and the multiplicity is 
expressed via the number of intervals of monotonicity 
and number of jumps of the symbol, see Theorem \ref{PWC:thm}.

To conclude the introduction we make some notational conventions. 
The unit circle ${\mathbb T}$ is equipped with the normalized 
Lebesgue measure $d{\bf m}(\z) =   (2\pi i\z)^{-1} d\z$ where 
$ \z\in\mathbb T$. For any $\z_1, \z_2\in \mathbb T$, we denote by 
$(\z_1, \z_2)$ the open arc joining $\z_1$ and $\z_2$ 
counterclockwise.   
For general information on functions analytic on the unit disk $\mathbb D$, 
we refer, for example, to the books  \cite{Duren} or \cite{Hof}. In particular,
the notation $\mathbb H^p = \mathbb H^p(\mathbb T)$, $p >0$, 
stands for the classical Hardy spaces. By $\|\cdot\|_p$ 
we denote the standard norm in $\mathbb H^p$.
The space $\mathbb H^2 \subset L^2(\mathbb T)$ is considered as  
a subspace of $L^2(\mathbb T)$, with inner product 
\begin{align*}
(f, g) = \int\limits_{\mathbb T} f(\z) \overline{g(\z)} d{\mathbf m}(\z).
\end{align*}
The orhogonal projection onto $\mathbb H^2$ is denoted by $\mathbb P$. By $E(X)$ with a Borel set 
$X\subset\mathbb R$ we denote the 
spectral family of a  self-adjoint Toeplitz operator $T$. We also use the standard notation 
$E(\lambda) = E((-\infty, \lambda))$, $\lambda\in \mathbb R$. 
 For a set $B$,  we denote its closure by $\clos B$.

\textit{Throughout the paper we assume that $\omega$ is a non-constant function on $\mathbb T$. }

\section{ Toeplitz operators, their quadratic forms and resolvents}\label{toeplitz_forms:sect}

\subsection{ Basic definitions}

Let us first recall the precise definition of Toeplitz operators.  
 If $\omega\in L^\infty( {\mathbb T})$, 
then the  Toeplitz operator $T = T(\omega)$ is defined on the space  ${\mathbb H}^{2}$ 
by the formula
\begin {equation*}
 Tf = {\mathbb P} (\omega  f), \q f\in {\mathbb H}^{2}.
\end{equation*}
We always suppose that the symbol $\omega$ is real-valued, 
so that $T(\omega)$ is self-adjoint. 

If $\omega$ is unbounded, we define 
the operator $T(\omega)$  via the sesqui-linear form
\begin{align}\label{tform:eq}
T[f, g] = \int\limits_{\mathbb T} \omega(\z) f (\z) \ov{g(\z)} d{\bf m}(\z)
\end{align}
where $f,g$ are polynomials   (or  $f,g\in \mathbb H^\infty$). 
The form is well-defined under the following condition.
  
\begin{cond}\label{om:cond}
  \begin{enumerate}[{\rm(i)}]
 \item 
 $\omega$ is real valued and $\omega\in L^1(\mathbb T)$, 
 \item 
  $\gamma_1 = \textup{ess-inf}\ \omega>-\infty$.   
 \end{enumerate}
 \end{cond}

This condition  ensures that 
the form \e{tform:eq} is semi-bounded from below and, as we will see, it is closable.
 
 Let us introduce the Schwarz kernel
\begin {equation}
H(z) = \frac{1+z}{1-z} ,\q z\in \mathbb D.
\label{eq:HH}\end{equation}
Then  
\begin{align}\label{Poisson:eq}
\frac{1}{2\pi}\Re H(re^{i\theta})
  = \frac{1}{2\pi} \frac{1-r^2}{1-2r\cos\theta + r^2}=: \mathcal P(r, \theta), \q r < 1,
\end{align}
is the Poisson kernel. Obviously, 
\begin{align}\label{eq:Poisson1}
\mathcal P(r, \theta) > 0 \ \ \textup{and}\ \ \int_{-\pi}^{\pi}\mathcal P(r, \theta)d \theta=1, \q
\forall r <1.
\end{align} 
The function
\begin{align}
K_u(z) = \frac{1}{1-\overline u z},\q z\in \mathbb D,\q u\in \clos\mathbb D,
\label{eq:KK}\end{align}  
is known as the reproducing kernel.
It follows from the formula
\begin{align}\label{reprod:eq}
(f, K_u) = \int\limits_{\mathbb T} \frac{f(\z)}{1-u \overline{\z}} 
\ d{\bf m}(\z)
= f(u),\q f\in \mathbb H^2,
\end{align}
that the set $\{K_u, u\in \mathbb D\}$ is  total  in the space $\mathbb H^2$, 
that is, the set $\mathcal{K} = {\rm span}\,\{K_u, u\in \mathbb D\}$ 
is dense in this space. 
Clearly, the set $\mathcal{K}$ consists of all rational 
functions with simple poles lying in the exterior of $\clos\mathbb D$ 
and tending to zero at infinity.
Let us also note the identities  $K_u(z) =\overline{K_z(u)}$ and 
\begin{align}\label{toz:eq}
H(\overline u z) + H(v\overline z) = 2(1-\overline u v|z|^2 )  K_u(z) \overline{K_{  v}(z)},
\end{align} 
valid for all $z,u, v\in \mathbb D$.
 
For  
$\lambda< \gamma_{1}$, we introduce  an outer function (see, e.g., \cite{Hof}  for basic properties of such functions)
\begin{equation}
F_\lambda(z) 
= \exp\bigg(
\frac{1}{2}\int\limits_{\mathbb T} 
\ln \big(\omega(\z)-\lambda\big) H(z\overline\z) d{\bf m}(\z)
\bigg)
\label{eq:FF}\end{equation}
 of $z\in \mathbb D$, associated with the function 
$(\omega(\z)-\lambda)^{1/2}$  of $\z\in \mathbb T$. 
Since $\omega\in L^1(\mathbb T)$, we have $F_\lambda\in \mathbb H^2$ and 
\begin{equation}
\omega(\z)-\lambda    = |F_\lambda(\z)|^2,\q \textup{a.e.} \q \z\in\mathbb T.
\label{eq:Poiss}\end{equation}
Thus,  the quadratic form  \e{tform:eq} can be written as
\begin{align}\label{ts:eq}
T[f, g]  = ({\bf F}_\lambda f, {\bf F}_\lambda g) + \lambda(f,g),
\end{align}
where ${\bf F}_\lambda$ is the operator of multiplication by $F_\lambda$ 
 in $\mathbb H^2$. Since the operator ${\bf F}_\lambda$ is closed on the domain
$D({\bf F}_\lambda) : = \{f\in \mathbb H^2: F_\lambda f\in \mathbb H^2\}$, the form $T[f, f] $ is well defined and closed on the domain
$D[T]: =D({\bf F}_\lambda)$. This allows one (see, e.g., \cite[Ch. 10]{BS}) to define $T =T(\omega)$ via this form.

\begin{definition}
 The self-adjoint operator $T$   is correctly defined
on  a dense set $D(T)\subset D[T]$
by the relation
\begin{equation} 
(T f, g) = T[f, g]  ,\q \forall f\in D(T), \q \forall g\in D[T].
\label{eq:TS}\end{equation}
 \end{definition}

 The operator $T$ is semi-bounded from below by $\gamma_{1}$. 
 Comparing equalities \e{ts:eq} and \e{eq:TS}, we see that
 \[
 ((T-\lambda) f,   g)=({\bf F}_\lambda f, {\bf F}_\lambda g) ,\q \forall f, g\in D(T).
 \]
 By the definition of the adjoint operator, it follows that
 ${\bf F}_\lambda f\in D({\bf F}^*_\lambda)$ for all $f\in D(T)$ and
 \begin{equation} 
( T-\lambda ) f= {\bf F}_\lambda^*{\bf F}_\lambda f,\q \forall f\in D(T),\ \lambda < \gamma_1.
\label{eq:TS1}
\end{equation} 
In view of \eqref{eq:Poiss},  
$|F_\lambda (\z)|^2\ge  \gamma_1-\lambda$ for all $\z\in{\mathbb T}$, whence 
$F_\lambda^{-1}\in \mathbb H^\infty$. This implies that for all $\lambda<\gamma_{1}$,
  \begin{align}\label{ran:eq}
\Ker {\bf F}_\lambda  =\{0\}, \q \Ran {\bf F}_\lambda = \mathbb H^2,
\end{align}
and the inverse ${\bf F}_\lambda ^{-1}$ is a 
bounded operator on $\mathbb H^2$. 
Since the operator ${\bf F}_\lambda$ is closed, its adjoint ${\bf F}_\lambda^*$ is 
densely defined and ${\bf F}_\lambda^{**}={\bf F}_\lambda$. Now \e{ran:eq} implies that
$\Ker {\bf F}_\lambda^*  =\{0\}$ and $\clos\Ran {\bf F}_\lambda^* = \mathbb H^2$. 
The inverse operator $({\bf F}_\lambda^*)^{-1}$ 
is defined on $\Ran {\bf F}_\lambda^*$ and (see, e.g., \cite[Theorem~3.3.6]{BS})
\[
({\bf F}_\lambda^*)^{-1} = ({\bf F}_\lambda^{-1})^*.
\] 
In particular, $({\bf F}_\lambda^*)^{-1}$ 
  extends to
 a bounded operator  on the whole space $\mathbb H^2$.

Recall that the function   $K_u$, $u\in \mathbb D$, is defined by formula \e{eq:KK}. According to \e{reprod:eq} we have
\begin{align*}
({\bf F}_\lambda f, K_u)  
= F_\lambda(u)( f, K_u), \q \forall f\in D({\bf F}_\lambda).
\end{align*}
Thus the adjoint operator ${\bf F}_\lambda^*$ acts  
on $K_u $ by the formula 
\[
({\bf F}_\lambda^* K_u )(z)= \overline{F_\lambda(u)} K_u(z),\q z\in{\mathbb D} ,
\] 
whence 
\begin{equation}
( {\bf F}_\lambda^* )^{-1}K_u  =( \overline{F_\lambda(u)})^{-1} K_u.
\label{eq:TS2a}\end{equation}

\subsection{Resolvent}

Here we find an explicit formula for the 
 sesqui-linear form $((T-~\lambda)^{-1} K_u, K_v)$ of the 
 resolvent of the operator  $T$ for all $u,v\in{\mathbb D}$. Suppose first that $\lambda<\gamma_{1}$.
 We proceed from factorization \e{eq:TS1} 
 which implies 
\begin{equation*}
(T-\lambda)^{-1} = {\bf F}_\lambda^{-1}({\bf F}_\lambda^*)^{-1}.
\end{equation*}
  Using also \e{eq:TS2a}, we find that
\begin{align*}
( (T-\lambda)^{-1} K_u, K_v)
= &\ (({\bf F}_\lambda^*)^{-1}K_u, ({\bf F}_\lambda^*)^{-1} K_v) \\[0.2cm]
= &\   (\overline{F_\lambda(u)})^{-1} 
({F_\lambda(v)})^{-1}  ( K_u, K_v )
= (1-\overline u v)^{-1} (\overline{F_\lambda(u)})^{-1} 
({F_\lambda(v)})^{-1}  .
\end{align*}
In view of definition \e{eq:FF} this yields
\begin{multline} 
((T-\lambda)^{-1}\  K_u, K_v)
\\
= \ (1-\overline u v)^{-1} 
\exp\bigg(
- \frac{1}{2}\int\limits_{\mathbb T} 
\ln \big(\omega(\z) - \lambda\big)
\big( H(v\overline\z) +  H(\overline{u}\z)
\big) d\bf m(\z)
\bigg).
\label{eq:resolvent}
\end{multline} 
Here $\lambda < \gamma_1$, 
but, 
by analyticity of both sides,  this equality extends to complex $\lambda$.  
We have chosen the principal branch of the logarithm: $\arg \big(\omega(\z) - \lambda\big) = 0$ 
for $\lambda< \gamma_{1}$. Then
$\arg \big(\omega(\z) - \lambda\big)\in (-\pi,0)$ for $\Im \lambda>0$ and $\arg \big(\omega(\z) - \lambda\big)\in (0,\pi)$ for $\Im \lambda<0$. 
In particular, $\arg \big(\omega(\z) - \lambda \mp i0\big)= \mp \pi i$ if $\lambda>\gamma_{2}$.
Note also that
\[
\int\limits_{\mathbb T} 
\big( H(v\overline\z) +  H(\overline{u}\z)
\big) d{\bf m}(\z)= 2, \q \forall u,v\in {\mathbb D}.
\]
This implies that the limit values of the right-hand side of \e{eq:resolvent} for $\lambda+i 0$ and $\lambda-i 0$ are the same if $\lambda> \gamma_{2}$, and hence this function is analytic in the half-plane  $\Re z>\gamma_2)$.

Let us state the result obtained.

 \begin{proposition}\label{res}
 Let  Condition~$\ref{om:cond}$ be satisfied, and let the functions $K_{u} (z)$ be defined by formula \e{eq:KK}. Then formula 
\e{eq:resolvent} is true for all $u,v\in {\mathbb D}$ and all $\lambda$
in the complex plane with 
a cut along $[\gamma_1, \gamma_2]$.
\end{proposition}

We emphasize that this result is not new; see \cite{CSW} and \cite{Ros1}, for original proofs. Our derivation is an expanded version of Rosenblum's argument from \cite{Ros2}. 

\section{Spectral properties of general Toeplitz operators}\label{genT:sect} 
  
\subsection{Absolute continuity}

We proceed from the following abstract result (see, e.g., 
\cite[Theorem~XIII]{RS4} or \cite[Proposition 1.4.2]{Yafaev}), 
which can be interpreted as a weak form of the limiting absorption principle. 
It shows that the
absolute continuity is a consequence of the 
existence of appropriate  boundary values of the resolvent.  

\begin{proposition}\label{Y:prop}
Let $A$ be a self-adjoint operator on a Hilbert space 
$\mathcal H$ with the spectral measure $E_{A} (\cdot)$,  
and let $X\subset \mathbb R$ be a compact interval. 
Suppose that for some element $g\in \mathcal H$ there is a number $p >1$ such that 
\begin{align*}
\sup_{\varepsilon\in (0, 1]} \int\limits_X 
|\Im ((A-\lambda-i\varepsilon)^{-1}g, g)|^p d\lambda < \infty.
\end{align*} 
Then the measure $(E_{A} (\cdot)g,g)$ is absolutely continuous on the interval $X$.
\end{proposition}
Let us return 
to Toeplitz operators $T = T(\omega)$. 
\textit{  Recall that $\omega(\z)$ 
is always assumed to be a non-constant function.}
  
 \begin{lemma} \label{repres}
Let Condition~$\ref{om:cond}$  be satisfied, and let $\mathcal P(r, \theta) $ be the Poisson kernel \eqref{Poisson:eq}.
For $z = re^{i\theta}$, $r <1$, $\theta\in (-\pi, \pi]$, set
\begin{align}
\mu(t) = \mu(t; z) 
= \int\limits_{ \substack{\tau\in (-\pi,\pi]\\
\omega(e^{i\tau})< t}
}\mathcal P(r, \theta-\tau) d\tau,
\q t\in {\mathbb R}.
\label{eq:AB}\end{align}
Then $\mu (t)$ is non-decreasing, $\mu(t) = 0$ for $t\leq \gamma_1$, $\mu(t) = 1$ for $t> \gamma_2$ and
\begin{align}\label{modpsi:eq}
(1-r^{2}) |((T-\lambda-i\varepsilon)^{-1}\  K_z, K_z)|= \exp\bigg(
-  \int_{\gamma_{1}} ^{\gamma_{2}}
\ln \big|t - \lambda-i\varepsilon\big| d\mu(t)
\bigg),
\end{align}
for any $\varepsilon\not = 0$.
\end{lemma}

\begin{proof} 
The equality $\mu(t) = 0, t\le \gamma_1$, is obvious, and the fact that 
$\mu(t) = 1$ for $t> \gamma_2$ is a consequence of \eqref{eq:Poisson1}. 

  It follows from formulas \eqref{Poisson:eq} and \e{eq:resolvent} that
\begin{multline}
(1-r^{2}) |((T- \lambda-i\varepsilon)^{-1}\  K_z, K_z)|
\\
= \exp\bigg(
-  \int\limits_{-\pi}^\pi 
\ln \big| \omega(e^{i\tau}) - \lambda-i\varepsilon\big|
 \mathcal P(r, \theta-\tau) d\tau
\bigg).
\label{eq:AA}
\end{multline}
Using the change of variables $t=\omega(e^{i\tau})$, we can (see, e.g., 
\cite[\S 39, Theorem C]{Halm}) rewrite \eqref{eq:AA} 
as \e{modpsi:eq}. 
\end{proof} 
 
 \begin{lemma} \label{reprX}
  Suppose that Condition \ref{om:cond} is satisfied. 
Then for all $t\in (\gamma_{1}, \gamma_{2})$  and all 
$z\in\mathbb D$, we have the strict inequalities
\begin{align*}
0<\mu(t; z) <1.
\end{align*}
\end{lemma}

\begin{proof} 
Let us check, for example, that $\mu(t) <1$. If $\mu(t)=1$ for some $t<\gamma_2$, then according to the definition \e{eq:AB} we have
\[
 \int\limits_{ \omega(e^{i\tau})\geq t}\mathcal P(r, \theta-\tau) d\tau
=0.
\]
Since 
$ \mathcal P(r, \theta-\tau)>0$, 
it follows that the Lebesgue measure 
$|\{e^{i\tau} \in {\mathbb T}:\omega(e^{i\tau})\geq t\}|=0$. However, this measure is positive for every $t<\gamma_2$  by the definition of $\gamma_{2}$. The inequality $\mu(t) >0$ can be verified quite similarly.
\end{proof}

Now we are in a position to establish a weak form of the limiting absorption principle for Toeplitz operators.

\begin{theorem} \label{integr:lem}
Let the symbol $\omega$ of a Toeplitz operator $T = T(\omega)$ 
satisfy Condition $\ref{om:cond}$. Then 
for any  compact interval $X\subset \mathbb R$ and any $z\in \mathbb D$  
there exists a number $p >1$ such that 
\begin{align}\label{2p:eq}
\sup_{\varepsilon\in (0, 1]}\int\limits_X 
|((T(\omega) - \lambda-i\varepsilon)^{-1}\  K_z, K_z )|^{p}\, d\lambda  < \infty.
\end{align}
\end{theorem}

\begin{proof} 
We proceed from Lemma~\ref{repres}.
Note first that $\gamma_{1}<\gamma_{2}$ because $\omega$ is non-constant.
Suppose that $X = [a,b]$ where $a<\gamma_{1}<b<\gamma_{2}$. 
Choose some $b_{0}\in [b,\gamma_{2}]$ and split the 
integral in \e{modpsi:eq} into 
two integrals -- over  $(\gamma_{1}, b_{0})$ and over $(b_{0},\gamma_{2}) $:
\begin{multline}\label{eq:XY}
(1-r^{2})^{p}|((T-\lambda-i\varepsilon)^{-1}\  K_z, K_z)|^p
\\
= \exp\bigg(
-  p\int_{b_{0}} ^{\gamma_{2}}
\ln \big|t - \lambda-i\varepsilon\big| d\mu(t)
\bigg)
 \exp\bigg(
-  p\int_{\gamma_{1}} ^{b_{0}}
\ln \big|t - \lambda-i\varepsilon\big| d\mu(t)
\bigg).
\end{multline}
If $t\in(b_{0},\gamma_{2})$, then 
\[
|t-\lambda -i\varepsilon|\geq t-\lambda \geq b_{0}-b \q\text{and hence}\q
-\ln |t-\lambda -i\varepsilon|\leq -\ln ( b_{0}-b).
\]
It follows that the first factor on the right-hand side of \e{eq:XY} 
is bounded by $(b_{0}-~b)^{p (\mu(b_{0})-1)}$
uniformly in $\varepsilon>0$. 

Next, consider the integral over 
$t\in(\gamma_{1}, b_{0})$. 
By Lemma~\ref{reprX}, we have $\mu(b_{0})<1$.  
Let us now apply Jensen's inequality to the normalized measure 
$\mu(b_{0})^{-1} d\mu(t)$ on $(\gamma_{1}, b_{0})$. Then
\[
 \exp\bigg(
- p \int_ {\gamma_{1}}^{b_{0}}
\ln \big|t - \lambda-i\varepsilon\big| d\mu(t )\bigg)\leq  \mu(b_{0})^{-1}
\int_ {\gamma_{1}}^{b_{0}}
\big| t - \lambda-i\varepsilon\big|^{-p \mu(b_{0})} d\mu(t ).
\]
  Therefore it follows from the equality \e{eq:XY} that
\[
\int\limits_{X} |((T-\lambda-i\varepsilon)^{-1}\ K_z, K_z)|^p d\lambda
\leq C \int\limits_{\gamma_{1}}^{b_{0}}
\bigg(\int\limits_{X}\big|t - \lambda\big|^{-p\mu(b_{0})}d\lambda\bigg) d\mu(t).
\]
The right-hand side here is finite as long as $p\mu(b_{0})<1$.

Thus, we have proved \e{2p:eq}
for $X = [a,b]$ where $a$ and $b$ are arbitrary numbers such that $a<\gamma_{1}<b<\gamma_{2}$. 
If $\gamma_{2}=\infty$, this concludes the proof. 
If $\gamma_{2}<\infty$, then we additionally have to consider intervals
$X = [a,b]$ such that $\gamma_{1}<a <\gamma_{2}< b$. 
Now we split the integral in \e{modpsi:eq} into 
two integrals over $(\gamma_1, a_{0}) $ 
where $ \gamma_{1} < a_{0}<a$ and over $(a_{0}, \gamma_{2})$.  
Similarly to the first part of the proof, 
the integral over $t\in(\gamma_1, a_{0})$ 
is bounded uniformly in $\varepsilon>0$. For $t\in( a_{0}, \gamma_2)$ 
we use the fact that $\mu(a_0)>0$ and apply Jensen's inequality again. 
\end{proof}

According to Proposition~\ref{Y:prop} it follows from
Theorem~\ref{integr:lem} that the  measures 
$(E(\ \cdot\ )K_z, K_z)$  are absolutely continuous on $\mathbb R$  
for all $z\in{\mathbb D}$. 
Therefore the measures $(E(\ \cdot\ )g, g)$ 
are also absolutely continuous  
for all $g\in\mathcal K$.  
Since the set $\mathcal K$ is dense in $\mathcal H$, we 
arrive at the following theorem.
 
\begin{theorem}\label{ac:thm}
Let the symbol $\omega$ of a Toeplitz operator $T = T(\omega)$ 
satisfy Conditions~$\ref{om:cond}$  and be non-constant.
Then the operator $T $ 
is absolutely continuous.   
\end{theorem}

In passing,  
 we note that for matrix-valued analytic symbols $\omega$, 
the limiting absorption principle was established in \cite{Coju} via the Mourre method.


\subsection{Spectral family}
Our next step is to find a convenient formula 
for the spectral family $E(\lambda)$ 
of the operator $T=T(\omega)$. We proceed  from Proposition~\ref{res} and use  the following elementary but important fact.
  
\begin{lemma}[\cite{Ros2}]\label{ac}
Under Condition~$\ref{om:cond}$, the inclusion
\begin{align}\label{l1:eq}
\ln |\omega(\ \cdot\ ) - \lambda|\in L^1  ({\mathbb T}) 
\end{align}  
holds  for  a.e. $\lambda\in \mathbb R$, and for this set of the points $\lambda$, the function 
$\ln|\omega(\ \cdot\ ) -\lambda - i\varepsilon|$ converges to 
$\ln |\omega(\ \cdot\ ) - \lambda|$ in $L^1 ({\mathbb T})$ as $\varepsilon\to 0$.
\end{lemma}  

\begin{proof}
Write
\begin{align*}
\int_{\mathbb T} \ln \big(\omega(\z) - \lambda - i\varepsilon\big)d{\mathbf m}(\z)
=   J(\lambda + i \varepsilon) +  \sigma 
\end{align*}
with
\begin{align*}
J(\lambda+i \varepsilon) =  &\ 
\int_{\mathbb T} \ln \big[
\big(\omega(\z) - \lambda - i\varepsilon\big)
- \frac{1}{2}\ln \big(\omega(\z)^2+1\big)
\big]d{\mathbf m}(\z),\\
\sigma = &\ \frac{1}{2}\int_{\mathbb T} \ln \big(\omega(\z)^2+1\big)
d{\mathbf m}(\z).
\end{align*}
Introducing the measure 
\begin{align*}
\nu(t) = \int_{\omega(\z)< t} d{\mathbf m}(\z),
\end{align*}
we can rewrite $J(\lambda+i\varepsilon)$ as 
 (see, e.g., 
\cite[\S 39, Theorem C]{Halm})
\begin{align*}
J(\lambda+i\varepsilon)
= &\ \int_{-\infty}^\infty \big[\ln \big[
\big(t - \lambda - i\varepsilon\big)
- \frac{1}{2}\ln \big(t^2+1\big)
\big]d\nu(t),\\
=&\ - \int_{-\infty}^\infty \bigg[\frac{1}{t-\lambda-i\varepsilon} 
- \frac{t}{t^2+1}\bigg] \nu(t)dt
\end{align*}
where we have integrated by parts.
The limit  as $\varepsilon\downarrow 0$ of the right-hand side exists and is finite for a.e. $\lambda\in\mathbb R$, and hence 
so does the limit
\begin{align}
\lim_{\varepsilon\downarrow 0}
{\rm Re}\,J(\lambda+ i\varepsilon)
= 
\lim_{\varepsilon\downarrow 0}
\int_{\mathbb T} \ln|\omega(\z) - \lambda - i\varepsilon| 
d{\mathbf m}(\z) -\sigma.\label{eq:monot}\end{align}
The function 
$\ln|\omega(\z)-\lambda-i\varepsilon |$ converges 
as $\varepsilon \to 0$ monotonically to $\ln|\omega(\z) - \lambda|$.
The Monotone Convergence Theorem now ensures that 
the limit on the right-hand side of \e{eq:monot} 
equals $\int_{\mathbb T} \ln|\omega(\z) - \lambda | d{\mathbf m}(\z)$
which is thus finite for a.e. $\lambda$.
\end{proof}
  

Recall that for every $z\in \mathbb D$, 
 formula \eqref{eq:FF} defines $F_\lambda(z)$ and $F_\lambda(z)^{-1}$  as  functions of $\lambda$ analytic in 
the complex plane with a cut along $[\gamma_1,  \gamma_{2}  ]$. Let us set
\begin{align*}
F_\lambda(z)^{-1} = \xi(z; \lambda) \exp\big( i A(z; \lambda)\big),
\end{align*}
where
\begin{align}\label{xi:eq}
\xi(z; \lambda) = \exp\bigg(-\frac{1}{2}\int_{\mathbb T} 
\ln|\omega(\z) - \lambda|H(z\overline\z) d{\mathbf m}(\z)\bigg)
\end{align}
and
\begin{align}\label{Ain:eq}
A(z; \lambda) 
= - \frac{1}{2} \int_{\mathbb T} 
\arg\big(\omega(\z) - \lambda \big) H(z\overline{\z}) d{\mathbf m}(\z).
\end{align}
Note that $\xi(z; \lambda) = \xi(z; \overline{\lambda})$ and $A(z; \lambda) = -A(z; \overline{\lambda})$. 
With the functions $\xi$ and $A$, 
representation \eqref{eq:resolvent}  can be rewritten  as follows:
\begin{multline}\label{resolvent1:eq}
((T-\lambda\mp i\varepsilon)^{-1} K_u, K_v)
= (1-\overline u v)^{-1} \\
\times
 \overline{\xi(u; \lambda \pm i\varepsilon)}
 \xi(v; \lambda \pm i\varepsilon) 
\exp\biggl(i\big(A(v\pm i\varepsilon) + \overline{A(u\pm i\varepsilon)}\big) 
\biggr).
\end{multline} 
 Now we can 
 describe the boundary values of the functions $\xi (z,\lambda)$ and $A (z,\lambda)$ on the cut. 
Introduce the subset (see Figure~1)
\begin{align}\label{eq:mult}
\Gamma(\lambda) = \{\z\in\mathbb T: \omega(\z) < \lambda\}
\end{align} 
of ${\mathbb T}$ 
and observe that 
\begin{align*}
\lim\limits_{\varepsilon\downarrow 0} 
\arg\,
(\omega(\z) - \lambda \pm i\varepsilon)
= 
\begin{cases}
\pm \pi, \q &\z\in \Gamma (\lambda) ,\\[0.2cm]
0, \q &\z\notin \Gamma (\lambda).
\end{cases}
\end{align*}
We emphasize that $\Gamma (\lambda)$ is defined up to a set of measure zero on ${\mathbb T}$.

The following assertion is a direct consequence of definitions \e{xi:eq},  \e{Ain:eq} and Lemma~\ref{ac}.

\begin{figure}
\resizebox{8cm}{!}{
\begin{tikzpicture}
\draw[very thick]  plot[smooth, tension=.7] coordinates {(0.5,1.5) (1.5,2.5) (2.5,2)};
\draw[very thick]  plot[smooth, tension=.7] 
coordinates {(2.5,-1.5) (3.5,-1.5) (4,-0.5) (4.5,1) (5,1.5) (5.5,1) (6,-0.5) (6.5,-1.5) (7.5,-1)};
\draw [very thick] plot[smooth, tension=.7] 
coordinates {(0.5,-1) (0,-0.5) (-0.5,1) (-1,1.5) (-2,1) (-2.5,-0.5) (-3,-1) (-4,-1)};
\draw (-4,1) -- (7.5,1)node[right]{$\lambda$};
\draw (-4,-0.5) -- (7.5,-0.5) node (v5) {};
\draw[dashed] (-2,1) -- (-2,-0.5) node (v1) {};
\draw[dashed] (-0.5,1) -- (-0.5,-0.5) node (v2) {};
\draw[dashed] (0.5,1.5) -- (0.5,-1);
\draw[dashed] (2.5,2) -- (2.5,-1.5);
\draw[dashed] (4.5,1) -- (4.5,-0.5) node (v3) {};
\draw[dashed] (5.5,1) -- (5.5,-0.5) node (v4) {}; 
\draw[double][ultra thick, blue] 
(-2,-0.5) -- (-4,-0.5);
\draw[double][ultra thick, blue] 
(-0.5,-0.5) -- (0.5,-0.5);
\draw[double][ultra thick, blue] (2.5,-0.5) --  (4.5,-0.5);
\draw[double][ultra thick, blue] 
(5.5,-0.5) -- (7.5,-0.5);
\draw[-latex] (1,-2.5) -- (-2.5,-1)node[pos=-0.25, below] {$\Gamma(\lambda)$};
\draw[-latex] (1.5,-2.5) -- (0,-1);
\draw [latex-](3.5,-1) -- (2,-2.5);
\draw[latex-] (6,-1) -- (2.5,-2.5);  
\end{tikzpicture}
}
\caption{Set $\Gamma(\lambda)$}
\end{figure}
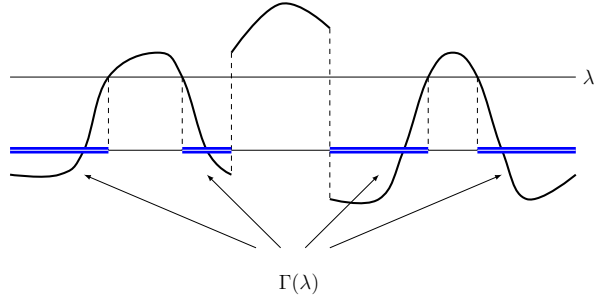


\begin{lemma}\label{limits:lem}
For a.e. $\lambda\in\mathbb R$ and all $z\in \mathbb D$, 
the functions $\xi(z; \lambda + i\varepsilon)$ 
and $A(z; \lambda + i\varepsilon)$ 
have limits as $\varepsilon\to \pm 0$. 
The limit 
\begin{align*}
\xi(z; \lambda) := \xi(z; \lambda+i0) = \xi(z; \lambda -i0)  
\end{align*}
is given by the formula \eqref{xi:eq} 
and 
\begin{align}\label{S:eq}
A(z; \lambda):= &\ A(z; \lambda+i 0) = - A(z; \lambda - i0)\notag\\
= &\ \frac{\pi}{2}\int_{\Gamma(\lambda)} H(z\overline{\z}) d{\mathbf m}(\z).
\end{align} 
 \end{lemma} 
The next fact follows from \eqref{resolvent1:eq} and Lemma \ref{limits:lem}. 
 
\begin{lemma}
For a.e. $\lambda\in\mathbb R$, we have the representation
\begin{align}\label{Sx:eq}
\lim\limits_{\varepsilon\downarrow 0}&\ ((T-\lambda\mp i\varepsilon)^{-1} K_u, K_v)\notag\\
= &\ (1-\overline u v)^{-1} 
 \overline{\xi(u; \lambda  )}
 \xi(v; \lambda  ) 
\exp\Big(\pm i\big(\overline{A(u; \lambda)} + A(v; \lambda)\big)\Big),
\end{align}
where $\xi(z; \lambda)$ and $A(z; \lambda)$ are given by \eqref{xi:eq} and 
\eqref{S:eq} respectively. 
\end{lemma}

Putting the representation
\e{Sx:eq} together with the Stone formula,  
\begin{multline*} 
2\pi i  \frac{d}{d\lambda}\big(E(\lambda) K_u, K_v\big)\\
=   \lim\limits_{\varepsilon\downarrow 0} \big(
((T-\lambda - i\varepsilon)^{-1} K_u, K_v) - 
((T-\lambda + i\varepsilon)^{-1} K_u, K_v)
\big) ,
\end{multline*}
we obtain

%

\begin{theorem}\label{specset:thm}
Suppose that $\omega$ satisfies Condition \ref{om:cond}. 
Define the functions $\xi (z; \lambda) $ and $A(z; \lambda)$ 
by formulas \e{xi:eq} and  \e{S:eq}, respectively. 
Then 
for all $u, v\in \mathbb D$ and a.e. $\lambda\in \mathbb R$, 
the spectral family $E(\lambda)$ of 
the Toeplitz operator $T$ 
satisfies the formula
\begin{equation} \label{spectral:eq}
\frac{d}{d\lambda}\big(E(\lambda) K_u, K_v\big)
= \frac{1}{\pi} (1-\overline u v)^{-1} 
 \overline{\xi(u; \lambda  )}
 \xi(v; \lambda  )  \sin \Big( \ov{A(u; \lambda)}+A(v; \lambda)\Big).
\end{equation}
\end{theorem}

Recall that $\gamma_1$ and $\gamma_2$ are defined in \eqref{eq:sig}. 

\begin{corollary}\label{specset}
The  spectrum of 
the operator $T$ coincides with the interval $[\gamma_1, \gamma_2]$.
\end{corollary}
  
\begin{proof}  
By the definition \eqref{tform:eq}, we have
\[
\gamma_1 \| f\|^{2}\leq T[f,f]\leq \gamma_2 \| f\|^{2},
\]
whence $\sigma(T)\subset[\gamma_1, \gamma_2]$. 
Conversely, it follows from \e{S:eq} that  
\begin{align*}
A(0; \lambda) = \frac{\pi}{2}{\bf m} (\Gamma (\lambda) ) 
\end{align*}
for a.e. $\lambda\in (\gamma_1, \gamma_2)$. 
By \eqref{spectral:eq}, this means that 
\begin{align*}
\frac{d(E(\lambda)K_0, K_0)}{d\lambda}
=  \frac{1}{\pi}  | \xi (0; \lambda)|^2  
\sin \big(\pi {\bf m} (\Gamma (\lambda) ) 
\big) >0
\end{align*}
because $0<{\bf m} (\Gamma(\lambda )) <1$.
Consequently, $\lambda\in \sigma(T) $ so that $ [\gamma_1, \gamma_2]\subset \sigma(T)$.
\end{proof}

\section{ Toeplitz operators with finite spectral multiplicity}\label{diag:sect}
 
\subsection{Auxiliary functions} 
Here we fix an interval $\Lambda\subset (\gamma_1, \gamma_2)$ and assume the following condition.
 
\begin{cond}\label{mult:cond}
For  a. e.  $\lambda \in\Lambda $, the set $\Gamma(\lambda)$ defined by \eqref{eq:mult} 
is a union of finitely many 
open arcs, whose closures are pairwise disjoint:
\begin{align}\label{gammaf:eq}
 \Gamma(\lambda)=\bigcup_{j=1}^{m} \; (\alpha_j(\lambda), \beta_j(\lambda)), \q m<\infty.
 \end{align}
 \end{cond} 
 
 Our goal is to diagonalize the operator $T$ using formula \eqref{spectral:eq}. 
This will be done locally, on the interval $\Lambda$. 
In particular, we will see that the spectral multiplicity of the Toeplitz operator $T=T(\omega)$ on 
 $\Lambda$ equals $m$.
 
First we calculate the function $A(z; \lambda)$ defined by \eqref{S:eq}. 
We often omit the dependence of  various objects on $\lambda$.

\begin{lemma}
Under  Condition~\ref{mult:cond}, for all $z\in\mathbb D$ and a.e. 
$\lambda\in \Lambda$, we have the representation
\begin{align}\label{Sf:eq}
A(z;\lambda) = \frac{\pi}{ 2}{\bf m }(\Gamma(\lambda))
 + \frac{i}{ 2}\sum_{j=1}^{m}
\ln  \frac{1- z\ov{ \alpha_j(\lambda)} }{1- z \ov{\beta_j(\lambda)}},
\end{align}
where the function   $\ln(1 + u)$ is analytic for $u\in {\mathbb D}$ and
$\ln 1=0$.
\end{lemma}

\begin{proof} 
 It suffices to check \e{Sf:eq} for the case when 
 $\Gamma$ consists of only one arc and 
 then to take the sum of the  results obtained.  
 Let $\Gamma= (\alpha,\beta)$, $\alpha=e^{ia}$, $\beta =e^{ib}$ where $0<a<b<2\pi$. We have to check that
\begin{align}\label{intdm:eq}
\int_{a}^{b}\frac{e^{i\theta}+z}{e^{i\theta}-z}d\theta 
= b-a + 2 i \ln
  \frac{1- z e^{-ia}}{1- z e^{-ib}}.
\end{align}
The easiest way to do it is to 
observe that both sides equal zero for $a=b$ and 
that their derivatives, for example, in the variable $b$, coincide.
\end{proof}

\begin{corollary}
For all $z\in \mathbb D$ and a.e. $\lambda\in \Lambda$,  
\begin{align} 
L(z; \lambda) := & i \pi^{-1}\exp \big(-2 iA(z; \lambda)\big)
\nonumber\\
=& i \pi^{-1}\exp\big( -\pi i {\bf m}(\Gamma(\lambda))\big)
\prod\limits_{j=1}^m  
\frac{1- z\ov{ \alpha_j(\lambda)} }{1- z \ov{\beta_j(\lambda)}}.
\label{repr2:eq} 
\end{align}
\end{corollary}

In particular, we see that $L(z; \lambda)$ is an analytic function of  $z$ in the  complex plane $\mathbb C$ with simple poles at the points
$\beta_j(\lambda)$, $j=1,\ldots, m$.

Let us collect together some elementary  identities needed below.

\begin{lemma}
Let  $\Gamma = (\alpha, \beta)\subset {\mathbb T}$.  Then
\begin{equation}
\beta  \overline{\alpha} = \exp(2\pi i \mathbf m(\Gamma)),  
\label{eq:ML}\end{equation}
\begin{align}
\label{od1:eq}
i  e^{-\pi i{\bf m}(\Gamma)} (1-\beta \overline{\alpha}) = |\beta-\alpha|
\end{align}
and
\begin{align}
\label{od:eq}
e^{-\pi i{\bf m}(\Gamma)} (1-\z \overline{\alpha})(1-\z \overline{\beta})^{-1}
=  |\z-\alpha| \, |\z-\beta|^{-1}
\end{align}
for any $\z\notin (\alpha, \beta)$.
\end{lemma}

\begin{proof}
Let $\alpha=e^{ia}$, $\beta=e^{ib}$. Then, by the definition of the measure $\mathbf m(\Gamma)$, both sides of 
  \e{eq:ML} equal $ e^{i(b-a)}$.

According to \e{eq:ML} the left-hand side of \e{od1:eq} equals 
\[
2\sin \pi {\bf m}(\Gamma)=2\sin \frac{b-a}{2}=| e^{i(b-a)}-1|
\]
which coincides with its right-hand side.

Finally, \e{od:eq} follows from \e{od1:eq} if we apply it to the pairs $\alpha,\z$ and $\beta,\z$ instead of the pair $\alpha, \beta$ and take into account that
$(\alpha,\z)=(\alpha, \beta)\cup [\beta, \z)$.
 \end{proof}

  Recall that  $H(z)$ is the Schwarz kernel defined by formula
\e{eq:HH}.

\begin{lemma}
 For all $z\in\mathbb C$ and a.e. $\lambda \in \Lambda$, the function \e{repr2:eq}  admits a representation 
\begin{align}
L(z; \lambda) =  \sum\limits_{j=1}^m c_j (\lambda) H\big(z \ov{\beta_j (\lambda) }\big)
+ \frac{i}{\pi} \cos\big(\pi {\bf m}(\Gamma(\lambda) \big),
\label{repr1:eq}
\end{align}
where  
\begin{align}\label{cj:eq}
c_j (\lambda)= \frac{1}{2\pi}\Big(\prod_{l= 1}^m 
|\beta_l (\lambda)-\alpha_l (\lambda)|\Big)\, \Big(\prod_{l\not = j}|\beta_j (\lambda)-\beta_l (\lambda)|^{-1}\Big)> 0,
\q j = 1, 2, \dots, m.
\end{align}
\end{lemma}

\begin{proof} 
First we note that
\begin{equation}
L(z) = \sum_{j=1}^m c_j H (z\bar{\beta_j}) + a 
\label{eq:LH}\end{equation}
with some complex constants   $c_j $, $ j=1, 2, \dots, m$, and $a$. 
Indeed,  both sides of equality \e{eq:LH} are rational 
functions with the same simple poles at the points 
$\beta_{j}$ and both of them have finite limits at infinity. We have to show that $c_{j}$ is given by 
\eqref{cj:eq} and find an expression for the constant $a$.

The residue of the right-hand side of 
\e{eq:LH} at the point $z=\beta_j$ equals $-2 c_{j} \beta_{j}$. 
Calculating the  residue of the function  \e{repr2:eq} at this point
and using equality \e{eq:LH}, we find that
\begin{equation}
2 c_{j} =   i \pi^{-1}\exp\big( -\pi i {\bf m}(\Gamma)\big)(1- \beta_{j}\ov{ \alpha_j })
\prod\limits_{l\neq j} 
\frac{1- \beta_{j}\ov{ \alpha_l } }{1- \beta_{j} \ov{\beta_l }}.
\label{eq:LHx1}\end{equation}
Put $\Gamma_{j}= (\alpha_{j}, \beta_{j})$. According to \e{od1:eq} we have
\[
i(1-\beta_{j} \overline{\alpha_{j}}) = e^{\pi i{\bf m}(\Gamma_{j})} |\beta_{j}-\alpha_{j}|,
\]
and it follows from \e{od:eq} for $\alpha=\alpha_{l}$, $\beta=\beta_{l}$ and 
$\z=\beta_{j}, j\not = l,$ that
\[
  (1-\beta_{j} \overline{\alpha_{l}}) (1-\beta_{j} \overline{\beta_{l}})^{-1}
= e^{\pi i{\bf m}(\Gamma_{l})} |\beta_{l}-\alpha_{l}| \, |\beta_{j}-\beta_{l}|^{-1}.
\]
Substituting these expressions into the right-hand side of \e{eq:LHx1} and using that 
\[
{\bf m}(\Gamma)={\bf m}(\Gamma_1) +\cdots + {\bf m}(\Gamma_m), 
\]
we obtain formula \e{cj:eq}
for the coefficients $c_{j}$.

Equality \eqref{eq:LH} implies that
\begin{equation*}
L(0) = \sum_{j=1}^{m} c_j + a \q \mbox{and}\q L(\infty) = -\sum_{j=1}^{m} c_j + a,
\end{equation*}
whence
\[
2a= L(0)+L(\infty)  .
\]
It follows from \e{repr2:eq} that 
\[
L(0)=\frac{i}{\pi}e^{-\pi i{\bf m}(\Gamma)} \q \mbox{and}\q L(\infty)
= \frac{i}{\pi}e^{-\pi i{\bf m}(\Gamma)} \prod_{j=1}^{m}\beta_{j}\overline{\alpha_{j}} =\frac{i}{\pi}e^{\pi i{\bf m}(\Gamma)}
\]
where at the last step we used equality \eqref{eq:ML}. Therefore
\[
a=\frac{i}{\pi} \cos\big(\pi {\bf m}(\Gamma  \big).
\]
Substituting this expression into  \eqref{eq:LH}, we conclude the proof of \eqref{repr1:eq}.
\end{proof}

\subsection{Spectral family and eigenfunctions}

Now we are in a position to introduce eigenfunctions of Toeplitz operators and 
to rewrite Theorem~\ref{specset:thm} in their terms. 

 \begin{theorem}\label{multiple:thm}
 Let $\omega$ satisfy   Condition~\ref{om:cond}, and 
 let Condition~\ref{mult:cond} be satisfied on some interval $\Lambda\subset
 (\gamma_{1}, \gamma_2)$.
 For  $j=1,\ldots,m$ and
  a.e. $\lambda\in \Lambda$,
denote
\begin{equation}
\varphi_j(z; \lambda) = \rho_j (\lambda)  \xi(z; \lambda)
\big(1-z\ov{\beta_j (\lambda)}\big)^{-1}
 \prod_{l=1}^m \big(1-z \ov{\alpha_l (\lambda)}\big)^{-\frac{1}{2}}
\big(1-z\ov{\beta_{l}(\lambda)}\big)^{\frac{1}{2}},
\label{eq:EL2}
\end{equation} 
where
 the function $\xi(z; \lambda)$ is defined 
by \eqref{xi:eq} and $  \rho_j(\lambda) = \sqrt{c_j(\lambda)}$ with the numbers $c_j(\lambda)$ 
given by \eqref{cj:eq}. Then
\begin{equation}
\frac{d\big(E(\lambda) K_u, K_v\big) }{d\lambda}
= \sum_{j=1}^{m} 
\overline{\varphi_j(u;\lambda)}\varphi_j(v;\lambda)
\label{eq:EL}\end{equation}
for all $u,v\in{\mathbb D}$ and
  a.e. $\lambda\in \Lambda$.
\end{theorem} 

\begin{proof}
Let us proceed from representation \e{spectral:eq}.  Using notation \e{repr2:eq} we see that
    \begin{align}
 \sin (\ov{A(u)}+ A(v)) &=  (2i)^{-1} e^{i A(v)}  e^{- i \ov{A(u)}}  \Big(  e^{2 i \ov{A(u)}}-e^{-2i A(v)} \Big)
\nonumber\\
&=  2^{-1}\pi e^{i A(v)}  e^{- i \ov{A(u)}}  \Big(   \ov{L(u)}+ L(v) \Big)
\label{eq:Sin}
\end{align}
It follows from \e{repr1:eq} that
\begin{align*}
\ov{L(u)}+ L(v)
=  \sum_{j=1}^{m} c_j \big(
H(\bar{u} \beta_j) + H(v \bar{\beta_j}) \big),
\end{align*}
where according to \e{toz:eq}, 
\[
H(\bar{u} \beta_j)  + H(v \bar{\beta_j}) =  2(1-\overline u v) 
 \overline {K_{\beta_j}(u)} {K_{\beta_j}(v)}.
\]
Substituting these expressions   into \eqref{eq:Sin}, we find that
\begin{equation}
 \sin (\ov{A(u)}+ A(v)) 
 = \pi e^{i A(v)}  e^{- i \ov{A(u)}}  
 (1-\overline u v)   
 \sum_{j=1}^{m} c_j  \overline {K_{\beta_j}(u)} {K_{\beta_j}(v)} .
\label{eq:EL1x}
\end{equation} 
Putting together \eqref{spectral:eq} and \e{eq:EL1x}, we get representation \e{eq:EL} with 
\begin{equation}
\varphi_j(z) =  
  e^{-\pi i {\mathbf m}(\Gamma)/2}\rho_{j}  \xi(z) K_{\beta_{j}} (z)  e^{i A(z)}.
\label{eq:EL1}
\end{equation}
In view of the definition \e{eq:KK} and formula \e{Sf:eq}, the relation \eqref{eq:EL1} 
coincides with 
\eqref{eq:EL2}.  
Thus the representation \e{eq:EL} holds, as claimed.
\end{proof}

\begin{corollary}
For all $z\in{\mathbb D}$  and $j=1,\ldots,m$,  we have
\begin{equation}
\varphi_j(z; \cdot)\in L^{2} (\Lambda).
\label{eq:LL}\end{equation}
\end{corollary}

\begin{proof}
Indeed, it follows from \eqref{eq:EL} that
\[
\sum_{j=1}^{m} 
\int_{\Lambda}|\varphi_j(z;\lambda)|^{2}d\lambda=\int_{\Lambda}\frac{d\big(E(\lambda) K_z , K_z\big) }{d\lambda} d\lambda\leq \|K_z\|^{2},
\]
which implies \eqref{eq:LL}. 
\end{proof}

In view of the absolute continuity of $T$ established in Theorem 
\ref{ac:thm},  we can also state

\begin{corollary}
For any bounded function $g:\mathbb R\to \mathbb C$ 
with support in $\Lambda$, 
and any $u, v\in \mathbb D$, 
we have
\begin{align}\label{Multiple:eq}
\big(g(T) K_u, K_v\big)
= \int\limits_{\Lambda} g(\lambda) 
\sum\limits_{j=1}^m 
\overline{\varphi_j(u; \lambda)} \varphi_j(v; \lambda)\, 
 d\lambda.
\end{align}
\end{corollary}

Note the special case $m = 1$, i.e. 
$\Gamma(\lambda) =  (\alpha(\lambda),\beta(\lambda))$ 
for a.e. $\lambda\in \Lambda$. 
In this case 
\begin{align}\label{simplee:eq}
\frac{d\big(E(\lambda) K_u, K_v\big) }{d\lambda}
= \overline{\varphi(u;\lambda)}\varphi(v;\lambda)
\end{align}
for all $u,v\in{\mathbb D}$ and a.e. $\lambda\in \Lambda$, with the 
function 
\begin{align}\label{simplephi:eq}
\varphi(z;\lambda) = \rho(\lambda)  \big(1-z /\alpha(\lambda)\big)^{-1/2}
\big(1-z /\beta(\lambda)\big)^{-1/2} \xi(z;\lambda), 
\end{align}
where, according to \eqref{cj:eq}, 
\begin{align}\label{simplerho:eq}
\rho(\lambda) = \sqrt{\frac{1}{2\pi}\big|   \beta(\lambda)-\alpha(\lambda) \big|}. 
\end{align}

\begin{remark}
The functions $\varphi_j$ in 
the right-hand side of the representation
\eqref{eq:EL} are not defined uniquely. In particular, we can obtain a  representation similar to
\eqref{eq:EL} but with the roles of $\alpha_j$ and $\beta_j$ 
reversed.  Indeed, let $\widetilde\Gamma = \widetilde\Gamma(\lambda) 
= \mathbb T\setminus \Gamma(\lambda)$ be 
the complement of $\Gamma(\lambda)$, that is, 
\begin{align*}
\widetilde \Gamma = \bigcup_{j=1}^m (\beta_{j-1}, \alpha_j),\q \beta_{0}:=\beta_m,
\end{align*}
up to a set of measure zero. Define $\widetilde A(z)$ by the formula \eqref{S:eq} with $\Gamma$ replaced by
$\widetilde\Gamma$. Then $A(z) = \frac{\pi}{2} - \widetilde A(z)$, and hence
\begin{align*}
\sin\big(\overline{A(u)} + A(v)\big) = \sin\bigg(\overline{\widetilde A(u)} + \widetilde A(v)\bigg).
\end{align*} 
Therefore the representation \eqref{spectral:eq} holds with 
$A(\ \cdot\ )$ replaced by $\widetilde A(\ \cdot\ )$. 
Implementing this change throughout the proof of Theorem \ref{multiple:thm}, 
we obtain the representation 
of the form \eqref{eq:EL} with the functions $\widetilde\varphi_j$ given by 
\begin{multline*}
\widetilde\varphi_j(z; \lambda) = \widetilde\rho_j (\lambda)  \xi(z; \lambda)
(1-z\ov{\beta_j (\lambda)}))^{-  1}
 \prod_{l=1} ^{m}(1-z\ov{\beta_{l}(\lambda)})^{-\frac{1}{2}}
(1-z \ov{\alpha_l (\lambda)})^{\frac{1}{2}}
\end{multline*} 
and (cf.  \eqref{cj:eq}): 
\begin{align*}
\widetilde\rho_j = 
\sqrt{\frac{1}{2\pi}\Big(\prod_{l= 1}^m |\alpha_j-\beta_l|\Big)\, \Big(\prod_{l\not = j}|\alpha_j-\alpha_l|^{-1}\Big)}.
\end{align*}
 \end{remark}

\subsection{Diagonalization}
  Now we are in a position to construct a unitary operator
\begin{equation*}
\Phi_{\Lambda} : E_T(\Lambda){\mathbb H}^{2}\to L^{2}(\Lambda; {\mathbb C}^m),
\end{equation*}
such that
\begin{equation}
(\Phi_{\Lambda} T f)(\lambda)=\lambda (\Phi_{\Lambda} f) (\lambda),\q \lambda\in \Lambda.
\label{eq:FL1}\end{equation}
These formulas mean 
that $\Phi_\Lambda$ diagonalizes the operator $T E(\Lambda)$ and the spectrum of 
$T$ on the interval $\Lambda$ has multiplicity $m$. 
First we construct a bounded operator
\[
\Phi : {\mathbb H}^{2} \to L^{2}(\Lambda;  {\mathbb C}^{m}),
\]
which is defined
on the set $\{K_z, z\in \mathbb D\}$, total in $\mathbb H^2$, by the formula
\begin{align}\label{phiku:eq}
(\Phi K_z)(\lambda) = 
\big\{\overline{\varphi_j(z; \lambda)}\big\}_{j=1}^m,\q \forall z\in{\mathbb D}, \q \textup{a.e.}\q 
\lambda\in \Lambda,
\end{align} 
where  $\varphi_j(z; \lambda)$ are functions  \eqref{eq:EL2}. 
The norm and the inner product 
in the space $L^2(\Lambda, \mathbb C^m)$ are denoted ${\boldsymbol |}\cdot {\boldsymbol |}$ and
 $\langle\ \cdot\ , \ \cdot\ \rangle$, respectively. 

Definition \eqref{phiku:eq} 
allows us to rewrite relation \eqref{Multiple:eq} for the characteristic function 
 $\1_{X}(\lambda)$ of 
a Borel subset $X \subset\Lambda$ as 
\begin{align}\label{direct:eq}
(E(X)K_u,  K_v) = 
\langle \1_{X}\Phi K_u, \Phi K_v\rangle ,
\end{align}
whence
\begin{equation}
(E(X)f  ,g ) = 
\langle \1_{X}\Phi f, \Phi g\rangle 
\label{eq:ad}\end{equation}
for all $f, g\in   \mathcal K = {\rm span}\,\{K_z, z\in \mathbb D\}$. 
In particular, we have
\begin{align*}
{\boldsymbol |}\Phi f {\boldsymbol |} 
= \|E(\Lambda) f\|, \q\forall f\in   \mathcal K,
\end{align*}
so that $\Phi$ extends to a bounded operator on the whole space  ${\mathbb H}^{2}$. 
This operator is isometric on the subspace $E (\Lambda) {\mathbb H}^{2}$ and  equals 
zero on its orthogonal complement. 

The construction of the adjoint operator $\Phi^{*}$ is quite standard.

\begin{lemma}
For any $\mathbf g= (g_{1}, \ldots, g_{m})\in L^2(\Lambda; \mathbb C^m)$, the operator $\Phi^*: L^2(\Lambda;\mathbb C^m)\to \mathbb H^2$ 
is given by the formula
\begin{equation}
(\Phi^*\mathbf g)(z) = \sum\limits_{j=1}^m
\int\limits_{\Lambda} \varphi_j(z; \lambda) g_j(\lambda) d\lambda, \q z\in\mathbb D.
\label{eq:LL1}\end{equation}
\end{lemma}

\begin{proof}
We first note that the right-hand side here is well-defined according to \e{eq:LL}.
By the definition \eqref{phiku:eq}, we have
\begin{align*}
\langle \mathbf g, \Phi K_z\rangle  =\sum\limits_{j=1}^m\int\limits_{\Lambda} 
\varphi_j(z; \lambda) g_j(\lambda) d\lambda  
\end{align*}
for all $z\in \mathbb D$.
At the same time, using \eqref{reprod:eq} we find that 
\begin{align*}
\langle \mathbf g, \Phi K_z\rangle = ( \Phi^*\mathbf g, K_z) = (\Phi^*\mathbf g)(z). 
\end{align*}
Putting together these two equalities, we get \e{eq:LL1}.
\end{proof}

 
 Next, we verify the intertwining property.
 
\begin{lemma}
For any Borel subset $X\subset \Lambda$ 
we have
\begin{align}\label{borel:eq}
\Phi E(X) = \1_{X}\Phi.
\end{align}
\end{lemma}

\begin{proof} 
Observe that
\begin{multline*}
(\Phi E(X) - \1_{X}\Phi)^*(\Phi E(X) -\1_{X}\Phi)
\\
= \big( E(X)\Phi^*\Phi E(X) -E(X)\Phi^* \1_{X}\Phi  \big) +\big( - \Phi^* \1_{X}\Phi E(X)  + \Phi^* \1_{X}\Phi \big).
\end{multline*}
Both terms on the right are equal to zero because, 
according to  \eqref{eq:ad}, 
$\Phi^* \1_X\Phi = E(X)$ and, in particular,  $\Phi^*\Phi = E(\Lambda)$.
\end{proof}
%
%


Now we define the operator $\Phi_\Lambda: E(\Lambda)\mathbb H^2\to L^2(\Lambda; \mathbb C^m)$ 
by the formula
\begin{align}\label{philam:eq}
\Phi_\Lambda f = \Phi f,\ f\in E(\Lambda)\mathbb H^2.
\end{align}
This operator is isometric, and, due to \eqref{borel:eq}, it satisfies \eqref{eq:FL1}. 
It remains to show that the mapping $\Phi_\Lambda$ is surjective and hence unitary.

\begin{lemma}
We have $\Ran \Phi_{\Lambda} = L^2(\Lambda; \mathbb C^m)$. 
\end{lemma}

\begin{proof}
Supposing the contrary, 
we find 
 an element $\mathbf g= (g_{1}, \ldots, g_{m})\in L^2(\Lambda; \mathbb C^m)$ such that
$\langle \mathbf g, \Phi (E(X)K_z)\rangle  = 0$ 
for all Borel sets $X\subset \Lambda$ and all $z\in \mathbb D$. 
By the definition \eqref{phiku:eq} and by \eqref{borel:eq}, this rewrites as 
\begin{align*}
\int\limits_{X}  \sum\limits_{j=1}^m
g_j(\lambda) \varphi_j(z; \lambda)  
  d\lambda=0.
\end{align*}
Since $X$ is arbitrary, this 
implies that 
\[
\sum\limits_{j=1}^m
g_j(\lambda) \varphi_j(z; \lambda) = 0
\]
 for   a.e. $\lambda\in\Lambda$.
 It now follows   from definition \e{eq:EL1} that
 \[
\xi(z; \lambda) e^{iA(z; \lambda)}\sum\limits_{j=1}^m \rho_j  
g_j(\lambda)  {(1- z \ov{\beta_j(\lambda)})^{-1}} = 0.
\]
Since by definition \e{xi:eq} $\xi(z; \lambda) \neq 0$, 
and  $\exp(iA(z; \lambda))\not = 0$, we see that
\begin{equation}
\big(1 - z\overline{\beta_l (\lambda)}\big) \sum\limits_{j=1}^m \rho_j  
g_j(\lambda)   {(1- z \ov{\beta_j(\lambda)})^{-1}} = 0
\label{eq:tot}
\end{equation}
 for all $z\in \mathbb D$, all 
 $l = 1, \ldots, m$ and a.e. $\lambda\in\Lambda$. 
 Passing here to the limit
 $z\to\beta_l(\lambda)$, we find that $\rho_l g_l(\lambda)=0$  
 and hence $g_l(\lambda)=0$ because $\rho_l\neq 0$, for all $l=1,\ldots, m$.
 \end{proof}
 
 \begin{remark}
The set of $\lambda\in \Lambda$  where relation \e{eq:tot} is satisfied, might depend on $z$. 
This difficulty is however inessential for our construction  because 
it suffices to work on a subset $\mathcal{K}_{0}\subset \mathcal{K}$ of linear combinations of functions
 $K_{z}$ with rational $z\in \mathbb D$. 
 The set $\mathcal{K}_{0}$ remains dense in ${\mathbb H}^{2}$, but it is countable. 
 Therefore  \e{eq:count} is satisfied on a set of $\lambda\in \Lambda$ of full 
 measure which is independent of rational $z$. Then it suffices to pass in \e{eq:count} to the limit
 $z \to\beta_l(\lambda)$ by a sequence of rational $z$.
 \end{remark}
 
Let us summarize the results obtained.

\begin{theorem}\label{specuni:thm}
 Let $\omega$ satisfy Condition~\ref{om:cond}, 
 and let  Condition~\ref{mult:cond} be satisfied on some interval 
 $\Lambda\subset (\gamma_{1}, \gamma_2)$ with some finite number $m$. Define the operators  
 $\Phi: {\mathbb H}^2\to L^{2} (\Lambda; {\mathbb C}^m)$ and 
 $\Phi_\Lambda: E(\Lambda){\mathbb H}^2\to L^{2} (\Lambda; {\mathbb C}^m)$  
 by formulas 
 \e{phiku:eq} and \eqref{philam:eq} respectively. 
Then 
\begin{align*}
\Phi_{\Lambda}^*\Phi_{\Lambda} = E(\Lambda),\q \Phi_{\Lambda}\Phi_{\Lambda}^* = I
\end{align*}
and 
\[
\Phi_{\Lambda} E(X) = \1_{X}\Phi_{\Lambda}
\]
for all Borel subsets $X\subset \Lambda$.
Thus the spectral representation of the operator $T E(\Lambda)$ is realized by the 
unitary operator $\Phi_\Lambda$, and the spectrum of $T$ on the interval $\Lambda$ 
has multiplicity $m$. 
\end{theorem}

\subsection{The operator $\Phi$}
  
The operator $\Phi$ defined by formula \e{phiku:eq} 
on the dense set $\mathcal K$, can be extended to the whole space 
$\mathbb H^2$ by the following natural formula.

\begin{proposition}
Let Condition~\ref{mult:cond} hold, and let $\varphi_j(z; \lambda)$ be  functions \e{eq:EL2}.
For all $f\in \mathbb H^2$ and $r\in (0, 1)$, consider the integral 
\begin{align}\label{Phir:eq}
(\Phi^{(r)} f)_j(\lambda) = 
\int\limits_{\mathbb T} 
f(\z)\overline{\varphi_j(r\z; \lambda)} 
 d{\bf m}(\z), \q j=1, 2, \dots, m,\q \text{a.e.}\q \lambda\in \Lambda.
 \end{align}
Then 
   \begin{enumerate}[{\rm(i)}]

\item  
The formula \eqref{Phir:eq} defines a contraction  
$\Phi^{(r)}: \mathbb H^2\to L^2(\Lambda; \mathbb C^m)$;

\item 
The family $\Phi^{(r)}$ converges strongly to the operator $\Phi$ as $r\to1$.
\end{enumerate}
\end{proposition}
 
\begin{proof}  
By definition \e{Phir:eq}, for all
   $f\in \mathbb H^2$ we have
\begin{align}\label{st:eq}
{\boldsymbol |} \Phi^{(r)} f {\boldsymbol |}^2 =\int \int   
f(\z)  \overline{f(\eta)} 
 \bigg[\sum\limits_{j=1}^m\int\limits_{\Lambda} 
\overline{\varphi_j(r\z; \lambda)} \  
\varphi_j(r \eta; \lambda)
d\lambda \bigg]
d{\bf m}(\z) d{\bf m}(\eta),
\end{align}
where all integrals without indication of the domain are taken over $\mathbb T$.
By \eqref{phiku:eq}, \eqref{direct:eq}, the integral in the square brackets equals
\begin{align*}
\langle \Phi K_u, \Phi K_v\rangle=
(E(\Lambda)K_u, K_v)
= \int  (E(\Lambda)K_u)(\sigma) \overline{K_v(\sigma)} d{\bf m}(\sigma), 
\end{align*}
where $ u = r \z$, $ v = r \eta$.
Taking into account that $\overline{K_{r\eta}(\sigma)} = K_{r\sigma}(\eta)$, and using 
\eqref{reprod:eq} 
we can now rewrite  
the right-hand side of \eqref{st:eq} as  
\begin{align*}
\int f(\z) \int(E(\Lambda) K_{r\z})(\sigma)&\  \overline{\bigg[\int f(\eta) \overline{K_{r\sigma}(\eta)}
d{\bf m}(\eta)\bigg]} d{\bf m}(\sigma)d{\bf m}(\z) \\[0.2cm]
= &\ \int f(\z) \bigg[\int (E(\Lambda) K_{r\z})(\sigma) \overline{f(r\sigma)} d{\bf m}(\sigma)\bigg]d{\bf m}(\z).
\end{align*}
With the notation $f_r(\sigma) = f(r\sigma)$, the integral in $\sigma$ equals
$(E(\Lambda) K_{r\z}, f_r) = ( K_{r\z}, E(\Lambda)f_r) = \overline{(E(\Lambda)f_r)(r\z)}$
where we have applied \eqref{reprod:eq} again. 
Thus it follows from \e{st:eq} that
\begin{align*}
{\boldsymbol |}\Phi^{(r)}f{\boldsymbol |}^2 = \int f(\z) \overline{(E(\Lambda)f_r)(r\z)}d{\bf m}(\z)
\le \|f\| \ \|(E(\Lambda) f_r)_r\|. 
\end{align*}
Since  $\| h_r\|\leq \| h\|$   for any $h\in \mathbb H^2$, 
the right-hand side does not exceed $\|f\|\ \|E(\Lambda)f_r\|\le \|f\|^2$.
Hence $\|\Phi^{(r)}\|\le 1$, as required. 

Proof of (ii).  
It suffices to check the strong convergence of $\Phi^{(r)}$ on the dense set $\mathcal K$. 
For $f = K_u$ and arbitrary $u\in\mathbb D$, we have 
\begin{align*}
\overline{(\Phi^{(r)} K_u)_j(\lambda)} = &\ 
\int 
\frac{1}{1- u \overline{\z}} \varphi_j(r \z; \lambda)
d{\bf m}(\z)\\[0.2cm]
= &\ 
\frac{1}{2\pi i} 
\int  
\frac{1}{\z-u } \varphi_j(r \z; \lambda)
d\z   = 
\varphi_j(r u; \lambda)
= \overline{(\Phi K_{ur})_j(\lambda)},
\end{align*}  
where we have used the definition \eqref{phiku:eq}. 
Since $K_{ur}\to K_u$  in $\mathbb H^2$ as $r\to 1$ and $\Phi$ is bounded,  we conclude that
 $ \Phi^{(r)} K_u \to \Phi K_u$ for all $u\in \mathbb D$. This completes the proof.
\end{proof}


\section{Eigenfunctions of Toeplitz operators}\label{eigen:sect}

Functions \e{eq:EL2} do not of course belong to the space ${\mathbb H}^{2}$. 
Nevertheless we check here that, in a natural sense, they satisfy
the equation  $T(\omega )\varphi_{j} (\lambda) = \lambda \varphi_{j} (\lambda)$. 
We recall that the numbers $  \rho_j(\lambda) = \sqrt{c_j(\lambda)}$ in \e{eq:EL2}  were defined 
by formula \e{cj:eq}, but in this section they are inessential.
 
Afterwards we also discuss 
two examples of Toeplitz operators 
with simple spectrum for which the eigenfunctions   
can be found explicitly. 

\subsection{Riemann-Hilbert problem}
Let us start with precise definitions. 
Let $\mathcal{ A}^{\rm ( int)}$ consist of functions 
$\varphi (z)$ analytic in the unit disk ${\mathbb D}$ 
and having radial limits $\varphi(\z_+) = \lim_{r\to 1-0}\varphi(r\z)$  
for a.e. $\z\in {\mathbb T}$. 
 Similarly, $\mathcal{ A}^{\rm (ext)}$ 
 consists of functions $\varphi^{\rm (ext)}(z)$ 
 analytic  in ${\mathbb C}\setminus \clos{\mathbb D}$, satisfying an estimate 
 $\varphi^{\rm (ext)} (z)= O (|z|^{-1})$
 as 
 $|z|\to\infty$ and having 
 limits $\varphi^{\rm (ext)} (\z_-) = \lim_{r\to 1+0}\varphi^{\rm(ext)}(r\z)$ 
 for a.e. $\z\in {\mathbb T}$.
 
\begin{definition}\label{RK}
 A function $\varphi\in\mathcal{A}^{\rm ( int)}$ 
 is a generalized eigenfunction of the 
 Toeplitz operator $T(\omega)$ corresponding to 
 a spectral point $\lambda$
  if the function $(\omega (\z)-\lambda)\varphi (\z)$ belongs to the class
 $\mathcal{A}^{\rm (ext)}$, that is, 
 there  exists  $\varphi^{\rm (ext)}\in\mathcal{ A}^{\rm (ext)}$ such that
 \begin{equation}
 (\omega (\z)-\lambda)\varphi (\z_+) =  \varphi^{\rm (ext)}(\z_-)
\label{eq:Ri-Hi}\end{equation}
 for a.e. $\z\in {\mathbb T}$.
 \end{definition}
 
This definition would reduce to the standard definition of 
an eigenfunction of the operator $T(\omega)$ if 
$\mathcal{A}^{\rm (int)}$ and 
$\mathcal{A}^{\rm (ext)}$ could be replaced by $ {\mathbb H}^{2}$  and 
${\mathbb H}^{2}_{-}$, respectively. 
Indeed, if there had been found functions 
$0\not= \varphi\in \mathbb H^2$ and 
$\varphi^{(\rm ext)}\in \mathbb H^2_-$, satisfying \eqref{eq:Ri-Hi}, then 
the function $\varphi$ would have satisfied the relation 
$T(\omega)\varphi = \lambda\varphi$, i.e. it would have been a proper eigenfunction of $T(\omega)$.

The relation \e{eq:Ri-Hi} 
looks like a standard homogeneous Riemann-Hilbert 
problem with the coefficient $\omega (\z)-\lambda$ but, 
in contrast to the classical presentation (see, e.g., the book \cite{Gahov}), the 
function $\omega$ is not assumed to be even continuous. 
However the worst complication comes 
from zeros of the function $\omega(\z)-\lambda$. 
As explained in \cite[\S 15]{Gahov}, even for smooth coefficients, 
the presence of roots makes the problem essentially more involved.  

We will check that the functions $\varphi_{j} (z,\lambda)$ 
defined by formula \e{eq:EL2} are generalized eigenfunctions 
of the Toeplitz operator $T(\omega)$ in the sense of Definition~\ref{RK}. 
In this subsection we fix some $\lambda\in (\gamma_{1}, \gamma_2)$ 
and assume that inclusion \e{l1:eq}   and the relation \eqref{gammaf:eq} are satisfied.

Let us first consider the function $\xi(z; \lambda)$ defined for 
all $z\in {\mathbb C}\setminus  {\mathbb T}$ by formula \e{xi:eq} and set
\begin{align}\label{eq:FFp}
Q(z; \lambda) = \int_{\mathbb T} 
\ln|\omega(\z) - \lambda|H(z\overline\z) d{\mathbf m}(\z) .
\end{align}
Then 
\begin{align}\label{eq:FF1}
\xi(z; \lambda)= \exp(-Q(z; \lambda) /2).
\end{align}
The integral in \e{eq:FFp} is convergent so that 
$Q(z; \lambda)$ is an analytic 
function of $z\in {\mathbb C}\setminus  {\mathbb T}$. Let us find the boundary values 
of $\xi(z; \lambda)$ on the circle $\mathbb T$.

\begin{lemma}\label{Priv}
Under assumption \e{l1:eq}
for  a.e. $\z\in {\mathbb T}$, there exist the limits
  \begin{equation}
\xi(\z_\pm; \lambda)=\sigma (\z,\lambda) |\omega(\z) - \lambda|^{\mp 1/2}  
\label{eq:www}\end{equation}
where the function 
  \begin{equation}
 \sigma (\z,\lambda) = \exp\bigg(\frac{1}{2}\int_{\mathbb T} 
\ln|\omega(\eta) - \lambda|  d{\mathbf m}(\eta) -  \frac{1}{2\pi i} v. p. \int_{\mathbb T} 
\ln|\omega(\eta) - \lambda| (\eta-\z)^{-1}  d \eta \bigg)
\label{eq:ss}\end{equation}
is the same for interior and exterior limits.
\end{lemma}

\begin{proof}
Since
  \begin{equation}
  H(z\overline\eta) d{\mathbf m}(\eta)= - d{\mathbf m}(\eta) + (\pi i)^{-1} (\eta-z)^{-1} d \eta,
  \label{eq:ssw}\end{equation}
  we see that
  \[
Q(z; \lambda) = -\int_{\mathbb T} 
\ln|\omega(\eta) - \lambda| d{\mathbf m}(\eta)  +\frac{1}{\pi i}\int_{\mathbb T} 
\ln|\omega(\eta) - \lambda| (\eta-z)^{-1}  d \eta.
\]
The first term on the right does not depend on $z$, and by the Sokhotski-Plemelj formula, we have
\[
\lim_{r\to 1\mp 0}\int_{\mathbb T} 
\ln|\omega(\eta) - \lambda| (\eta-r \z)^{-1}  d \eta= v. p. \int_{\mathbb T} 
\ln|\omega(\eta) - \lambda| (\eta-\z)^{-1}  d \eta \pm \pi i \ln|\omega(\z) - \lambda|
\]
for  a.e. $\z\in {\mathbb T}$.
This yields the limits $Q(\z_\pm; \lambda)$. 
In view of \e{eq:FF1}, we obtain relation \e{eq:www}.
\end{proof}

\begin{corollary}
For  a.e. $\z\in {\mathbb T}$, we have the  relation 
  \begin{equation}
 \xi(\z_-; \lambda) = |\omega(\z) - \lambda| \xi(\z_+; \lambda).
\label{eq:wwwc}\end{equation}
\end{corollary}

Lemma~\ref{Priv} leads also to the following result.

\begin{theorem}
Let $\sigma (\z,\lambda)$ 
be defined  by \e{eq:ss}. Then for $  j = 1,\ldots, m$ and a.e. $ \z\in {\mathbb T}$ 
the boundary values on the unit circle of the  functions  \e{eq:EL2} are given by the relation
  \begin{align*}
\varphi_{j}(\z_+; \lambda) = \rho_{j} (\lambda) 
\sigma (\z,\lambda) &\ |\omega(\z) - \lambda|^{- 1/2} \big(1-\z\ov{\beta_j (\lambda)}\big)^{-1}\\
&\  \quad\quad\times\prod_{l=1}^m \big(1-\z_+ \ov{\alpha_l (\lambda)}\big)^{-\frac{1}{2}}
\big(1-\z_+\ov{\beta_{l}(\lambda)}\big)^{\frac{1}{2}}.
\end{align*}
\end{theorem}

 We will now define 
the function $\varphi_j^{(\rm ext)}(z; \lambda)$ by the formula
\begin{align}\label{eq:www1}
\varphi_j^{(\rm ext)}(z; \lambda) =  e^{-\pi i {\mathbf m}(\Gamma)/2}\rho_j(\lambda) 
\xi(z; \lambda) (1-z\ov{\beta_{j}(\lambda) })^{-1}  e^{i A(z; \lambda)},\q |z|>1,
 \end{align}
 where $A(z; \lambda)$ is given by formula \e{S:eq}.
 Equations \e{eq:EL1} and \e{eq:www1} look the same 
 but the first of them applies for $|z|<1$ while the second one -- for $|z|>1$.
 Let us calculate the function $e^{i A(z)}$ for $ |z|>1$. Instead of \e{intdm:eq} we now have the identity
 \begin{align}\label{eq:tw}
\int_{a}^{b}\frac{e^{i\theta}+z}{e^{i\theta}-z}d\theta 
= a -b+ 2 i \ln
  \frac{1-  e^{ia}/z}{1-  e^{ib}/z},\q |z|>1,
 \end{align}
whence
\[
  \int_{\Gamma}  H(z\overline\z) d{\mathbf m}(\z)= -{\bf m} (\Gamma ) + i\pi^{-1}\sum_{l=1}^{m} \ln   \frac{1-  \alpha_l/z}{1- \beta_l/z}.
\]
According to  \e{S:eq} this yields
\begin{align*}
e^{i A(z)}=e^{-\pi i \mathbf m(\Gamma)/2}\prod_{l=1}^{m} (1- \alpha_l/z)^{-1/2} (1-\beta_l/z)^{1/2}.
 \end{align*}
In view of \e{eq:ssw}  it follows from the Sokhotski-Plemelj formula that
  \[
 A(\z_{+}) - A(\z_{-})=\pi \; \mbox{for} \;\z\in \Gamma\q
 \mbox{and}\q  A(\z_{+}) - A(\z_{-})=0 \; \mbox{for} \;\z\not\in \Gamma.
 \]
 whence
 \begin{equation*}
 e^{i A(\z_{+})}
 = 
\sign\big(\omega(\z) - \lambda\big)  e^{i A(\z_{-})}, \ {\rm a.e.}\  \z\in\mathbb T.
\end{equation*}
 Putting together this equality with
\e{eq:wwwc}, we  obtain the following result. Recall that the function $\xi(z; \lambda)$ was defined by relations \e{eq:FFp} and 
 \e{eq:FF1} for 
all $z\in {\mathbb C}\setminus  {\mathbb T}$.
  
  \begin{theorem}\label{R-H}
 Let $\omega$ satisfy   Condition~\ref{om:cond}, and 
 let   \e{l1:eq} and \eqref{gammaf:eq} hold for 
 some point $\lambda\in(\gamma_{1}, \gamma_2)$. Then
 for  all $j=1,\ldots,m$,  the equality 
  \begin{equation*}
(\omega (\z )-\lambda)\varphi_{j} (\z_+,\lambda)=  \varphi^{\rm (ext)}_{j} (\z_-,\lambda),
\end{equation*}
 is 
 satisfied with the functions $\varphi_j(z; \lambda)$ 
 and $\varphi^{\rm (ext)}_{j}(z; \lambda) $ defined by formulas \e{eq:EL2} and 
\begin{align*}
\varphi_j^{(\rm ext)}(z; \lambda) = \rho_j(\lambda) e^{-\pi i {\bf m} (\Gamma(\lambda))}
&\ \xi(z; \lambda) (1-z\ov{\beta_{j}(\lambda) })^{-1}\notag \\
&\ \times  \prod_{l=1}^{m} (1- \alpha_l (\lambda)/z)^{  -1/2} 
(1-\beta_l (\lambda)/z)^{ 1/2}, 
 \end{align*}
  respectively. 
  \end{theorem} 
  
  \begin{corollary}
Let $m=1$. Then $\varphi (z; \lambda)$ is given by formula  \e{simplephi:eq} and
  \begin{align}
\varphi^{(\rm ext)}(z; \lambda) 
= -\rho(\lambda)&\ 
  e^{-\pi i {\bf m} (\alpha(\lambda),\beta (\lambda))}
\beta z^{-1}
\xi(z; \lambda) \notag  \\
&\ \times(1- \alpha (\lambda)/z)^{ -1/2} (1-\beta (\lambda)/z)^{ -1/2}.
\label{eq:m=1}
\end{align}
\end{corollary} 
  
\begin{remark}
In the construction above, we used only that
the functions $\varphi_j(z; \lambda)$ 
are analytic in $\mathbb D$ and have boundary values on $\mathbb T$ for a.e.  
$\z\in\mathbb T$. In the next subsection we will see that, 
actually, $\varphi_j(\cdot; \lambda) \in {\mathbb H}^{p}$ 
for every $p<1$ and a.e. $\lambda\in (\gamma_{1}, \gamma_2)$.
\end{remark} 
  
\begin{remark}
In the spirit of Definition \ref{RK}, Theorem \ref{R-H} states that $\varphi_j, j=1, 2, \dots, m$, 
are generalized 
eigenfunctions of the operator $T(\omega)$. We certainly do not claim that the pairs 
$\varphi_j, \varphi_j^{(\rm ext)}$ give all solutions of the Riemann-Hilbert problem \eqref{eq:Ri-Hi}. 
For example,  other solutions can be obtained by multiplying 
the functions $\varphi_j (z)$ and $\varphi_j^{(\rm ext)}(z )$ constructed above by a common factor $ (z-\z_{0})^{-n}$ where the point 
$\z_{0}\in{\mathbb T}$
is arbitrary and $n=1,2,\ldots$.
\end{remark} 
  
  \subsection{  Uniform estimates near the unit circle}
The first assertion supplements Lemma~\ref{Priv}. Its proof does not require Condition~~\ref{mult:cond}.

\begin{lemma}
Let the function $\xi(z) = \xi(z; \lambda)$ be   defined   for a.e. $\lambda\in{\mathbb R}$  by \eqref{xi:eq}. Then for any $p <2$ and any bounded interval 
$X\subset  {\mathbb R}$ the estimate 
\begin{equation}
\sup_{z \in {\mathbb D} }  \int_{X} |\xi(z; \lambda)|^p   d\lambda <\infty
\label{eq:hz}\end{equation}
holds. 

\end{lemma}

\begin{proof} 
According to \eqref{Poisson:eq} and \eqref{xi:eq}, for $z = re^{i\theta}$ we have 
\begin{align*}
|\xi(z; \lambda)| 
= \exp\biggl(
-\frac{1}{2}\int_{-\pi}^\pi \ln|\omega(e^{i\tau}) - \lambda| \mathcal P(r, \theta - \tau) d\tau
\biggr).
\end{align*}
Therefore, arguing as in the proof of Lemma \ref{repres}, we obtain 
\begin{align*}
|\xi(z; \lambda)|
= \exp\biggl(-\frac{1}{2}\int_{\gamma_1}^{\gamma_2} 
\ln |t-\lambda| d\mu(t; z)\biggr), 
\end{align*}
with the measure $\mu(t; z)$ defined in 
\eqref{eq:AB}. 
Since $\mu(t; z)$ is normalized, it follows from Jensen's inequality that for any $p >0$ we have 
the bound
\begin{align*}
|\xi(z; \lambda)|^p\le \int\limits_{\gamma_1}^{\gamma_2} |t-\lambda|^{-\frac{p}{2}}d\mu(t; z)
\end{align*}
uniformly in $z\in \mathbb D$. Integrating it over a finite interval $ X \subset \mathbb R$, we see that
\begin{align}\label{eq:hz1}
\int\limits_{X} |\xi(z; \lambda)|^p d\lambda 
\le \int_{\gamma_1}^{\gamma_2} \biggl( \int\limits_{X} |t-\lambda|^{-\frac{p}{2}}
d\lambda\biggr)d\mu(t; z)
  \le C_p.
\end{align}
%
 with a constant $C_p = C_p(X)>0$, if $p <2$. This leads to \eqref{eq:hz}. 
\end{proof}

\begin{corollary}\label{cor:xi}
For every $p <2$ the function 
$\xi(\ \cdot\ ; \lambda)$ belongs to $\mathbb H^p$, a.e. $\lambda$, and 
its norm, as a function of $\lambda$ belongs to $L^p_{\rm loc}$.
\end{corollary}

\begin{proof}
Set
\[
M_p(r; \lambda) = \bigg[\int\limits_{\mathbb T} 
|\xi(r\z; \lambda)|^p d{\mathbf m}(\z)\bigg]^{\frac{1}{p}}.
\]
Integrating the inequality \eqref{eq:hz1} where $z = r\z$, $r\in (0, 1)$,  
 over $\z\in\mathbb T$ and
 exchanging the order of integration, we see that
 \begin{align}
\int\limits_{X} M_p(r; \lambda)^p d\lambda
\le C_p.
\label{eq:MH}\end{align}
According to \cite[Ch.1, Theorem 1.5]{Duren}, the function $M_p(r; \lambda)$ is non-decreasing in $r\in (0, 1)$. 
Therefore,  by the Monotone Convergence Theorem, the bound \e{eq:MH} remains true for
the limit $M_p(1; \lambda) := \lim_{r\to 1-0} M_p(r; \lambda)$.
Thus $M_p(1; \lambda)<\infty$, a.e. 
$\lambda$, so that $\xi(\ \cdot\ ; \lambda)\in \mathbb H^p$ a.e. $\lambda$, 
and its norm $M_p(1; \ \cdot\ )$ is in $L^p_{\rm loc}$.
\end{proof}
 
Now we consider the eigenfunctions of the operator $T$.

\begin{theorem}
 Let Condition \ref{mult:cond} be satisfied for some interval 
$\Lambda\subset (\gamma_1, \gamma_2)$. Then the function $\varphi_j(\ \cdot\ ; \lambda)$ 
defined by  formula \e{eq:EL2}, belongs to the space
$\mathbb H^p$  for every $p <1$ and a.e. $\lambda\in \Lambda$.  
\end{theorem}

\begin{proof}
Denote
\begin{align*}
h_j(z; \lambda) = \big(1-z\ov{\beta_j (\lambda)}\big)^{-1}
 \prod_{l=1}^m \big(1-z \ov{\alpha_l (\lambda)}\big)^{-\frac{1}{2}}
\big(1-z\ov{\beta_{l}(\lambda)}\big)^{\frac{1}{2}},
\end{align*}
so that $\varphi_j(z; \lambda) = \rho(\lambda) \xi(z; \lambda) h_j(z; \lambda)$. 

Fix a $p <1$. 
Then by Corollary \ref{cor:xi}, 
$\|\xi(\ \cdot\ ;\lambda)\|_q<\infty, q = 2p<2$, a.e. $\lambda\in\Lambda$. 
Furthermore, it follows directly from the definition that also 
$\|h_j(\ \cdot\ ;\lambda)\|_{q} <\infty$, a.e. $\lambda\in \Lambda$. Thus, 
on a subset of $\Lambda$ of full measure, we have, by H\"older's inequality,  
\begin{align*}
\|\varphi_j(\ \cdot\ ; \lambda)\|_p\le \rho_j(\lambda)
\|\xi(\ \cdot\ ; \lambda)\|_q \|h_j(\ \cdot\ ; \lambda\|_q<\infty,
\end{align*} 
as required. 
\end{proof}

In contrast to Corollary \ref{cor:xi}, we cannot say anything about integrability of the 
norms $\|\varphi_j(\ \cdot\ ; \lambda)\|_p$, 
  $q = 2p$, 
in $\lambda$, since   $\|h_j(\ \cdot\ ; \lambda)\|_{q}$ are not bounded if the 
singularities of the function $h_j(\ \cdot\ ; \lambda)$  merge as $\lambda$ varies.

\subsection{A smooth (regular) symbol}

Let us now discuss two explicit examples. Both of them were mentioned in \cite{Ros2}. 
First we consider the regular symbol 
 \begin{align*}
 \omega_{\rm r} (\z)= (\z+\z^{-1})/2.
  \end{align*} 
By Theorem \ref{ac:thm} and Corollary \ref{specset} 
the spectrum of $T_{\rm r} = T(\omega_{\rm r})$ is absolutely continuous 
and coincides with $[-1,1]$. 
For every $\lambda\in (-1, 1)$ we set
\begin{align}\label{eq:zz}
\alpha = \lambda + i\sqrt{1-\lambda^2},\q
\beta =\bar{\alpha}= \lambda -  i\sqrt{1-\lambda^2},
\end{align}
  so $\Gamma(\lambda) = (\alpha, \beta)$. 
In particular, we see that, by Theorem~\ref{specuni:thm}, the spectrum of $T_{\rm r}$ is simple. 

Let us calculate the function $\varphi = \varphi_{\rm r}$ defined in \eqref{simplephi:eq}. 

\begin{lemma}
Let $\omega = \omega_{\rm r}$, as defined above. 
Then for every $\lambda\in (-1, 1)$ we have 
\begin{align}\label{phir:eq}
\varphi_{\rm r}(z;\lambda) 
= \sqrt{\frac{2}{\pi}}\frac{(1-\lambda^{2})^{\frac{1}{4}}}{1-2\lambda z + z^{2}}.
\end{align}
\end{lemma}

\begin{proof}
Let us find  functions \eqref{eq:FFp} and \eqref{eq:FF1}  for the symbol $\omega=\omega_{\rm r}$. 
Note that $2|\omega_{\rm r}(\z) - \lambda| = |g(\z; \lambda)|$,   where 
\begin{align*}
g(\z; \lambda) = \z^2 - 2\lambda\z + 1 = (\z - \alpha)(\z - \beta).
\end{align*}
The function $g(z; \lambda)$ is 
analytic in $\mathbb D$, $g(z; \lambda)\neq 0$ and $g(0; \lambda)=1$ so that
the function $\ln g(z; \lambda)$, fixed by the condition 
$\ln g(0; \lambda) = 0$, is also analytic in $\mathbb D$ for 
every $\lambda\in (-1, 1)$. Let us rewrite \eqref{eq:FFp}    as
\begin{align*}
Q(z; \lambda)=\frac{1}{4\pi i} \int_{\mathbb T} \Big( \ln g(\z; \lambda)  
+  \ln \overline{ g(  \z; \lambda)} - 2 \ln 2\Big)\frac{\z+z} {(\z-z) \z} \, d\z.
\end{align*}
This integral can be easily calculated by residues at the points 
$\z=z$ and $\z=0$. 
The first term containing 
$\ln g(\z; \lambda)$ 
equals $ \ln g(z; \lambda)$ 
and the third term containing 
$-2\ln 2$ equals $-\ln 2$. 
In the second integral we use  that $\overline{ g(  \z; \lambda)}=g(\overline \z; \lambda)$ and make the change of variables $\z\mapsto \z^{-1}$:
\[
\int_{\mathbb T}\ln \overline{ g(  \z; \lambda)}\frac{\z+z} {(\z-z) \z} \, d\z 
= \int_{\mathbb T}\ln g( \z; \lambda)\frac{1+\z z} {(1-\z z) \z} \, d\z.
\]
This integral equals zero because $\z=0$ is the only pole
 of the integrand in $\mathbb D$ and $g(0,\lambda)=1$. It follows that 
\begin{align*}
Q(z; \lambda) = \ln g(z; \lambda)-\ln 2,
\end{align*}
and hence according to  \eqref{eq:FF1} 
\begin{align*}
\xi_{\rm r}(z; \lambda) =\sqrt { \frac{2 }{ 1-2\lambda z + z^{2}}}. 
\end{align*}
The coefficient $\rho(\lambda)$ in  \eqref{simplephi:eq} is found from 
\eqref{simplerho:eq} and \e{eq:zz}:
\begin{align*}
\rho(\lambda) = \sqrt{\frac{1}{2\pi} |\alpha-\beta|} = \pi^{-\frac{1}{2}} (1-\lambda^2)^{\frac{1}{4}}.
\end{align*}
Substituting these formulas into \eqref{simplephi:eq},  we obtain \eqref{phir:eq} .
\end{proof}

Clearly, equation \eqref{eq:Ri-Hi}  for function \eqref{phir:eq} is satisfied with
\[
\varphi^{\rm (ext)}_{\rm r}(z)= (2\pi)^{-\frac{1}{2}} (1-\lambda^2)^{\frac{1}{4}}z^{-1}.
\]

Recall (see formula (10.11.31) in the book \cite{BE}) that
the function $(z^2-2\lambda z+1)^{-1}$ is the generating function of 
the Chebyshev polynomials $U_n$ of 
second kind. This means that 
\begin{align*}
\frac{1}{z^2-2\lambda z+1} = \sum\limits_{n=0}^\infty U_n(\lambda) z^n,\quad \lambda\in (-1, 1), \q
z\in \mathbb D.
\end{align*}
Set $p_n(z) = z^n$.
Since 
\begin{align*}
\sum_{n=0}^\infty z^n {\overline v}^n = K_v(z), 
\end{align*} 
the relation \eqref{simplee:eq} together with \eqref{phir:eq} 
implies that 
\begin{align*}
\frac{d}{d\lambda} \big( E(\lambda)p_n, p_m\big)
= \frac{2}{\pi} \sqrt{1-\lambda^2}\  U_n(\lambda) U_m(\lambda), \quad \lambda\in (-1, 1),
\end{align*} 
for    all $n, m = 0, 1, \dots$. 
This formula agrees with the expression for the spectral projection 
of the discrete Laplacian on $\ell^2(\mathbb Z_+)$ 
(see, e.g.,  \cite[Corollary III.12]{YafJac}), 
which is unitarily equivalent to $T_{\rm r}$. 

Note that $T_{\rm r}$ is the unique 
(up to trivial changes of variables) Toeplitz operator that is unitarily equivalent to a Jacobi operator.

\subsection{A singular symbol}

The simplest singular symbol is given by  the indicator 
$\omega_{\rm s} (\z) = \1_{(\z_{1}, \z_{2})} (\z)$ 
of an arc $(\z_{1}, \z_{2})\subset {\mathbb T}$, $(\z_{1}, \z_{2})\neq {\mathbb T}$,
so that 
\begin{align*}
\Gamma(\lambda) = (\z_2, \z_1) \q \mbox{for all} \q \lambda\in (0, 1),
\end{align*} 
and hence the spectrum of the operator $T_{\rm s} = T(\omega_{\rm s})$ is simple and it fills the interval $[0, 1]$. 
The eigenfunctions of this operator are calculated in the next lemma. Recall that the branch of the function $\ln (1+z)$ analytic for $z\in {\mathbb D}$ is fixed by  the condition $\ln 1 =0$.

\begin{lemma}
 For all 
$\lambda\in (0, 1)$, the eigenfunctions \eqref{eq:EL2} of the  operator  $T_{\rm s}$ 
are given  by the formula
\begin{equation}
\varphi_{\rm s} (z;\lambda) = \rho(\lambda) e^{-\pi\sigma (\lambda){\mathbf m}(\z_1,  \z_2)} (1-z / \z_1)^{-1/2-i\sigma (\lambda)} 
(1-z/\z_{2})^{-1/2+i\sigma (\lambda)},
\label{eq:RS}
\end{equation}
where   
 \begin{equation*}
 \sigma (\lambda)= \frac{1}{2\pi}\ln(\lambda^{-1}-1)
\end{equation*}
 and the numerical coefficient $\rho(\lambda) $ is given by \eqref{simplerho:eq}.
\end{lemma}

\begin{proof} 
For the symbol $\omega = \omega_{\rm s}$,  the function \e{eq:FFp}
is given by
\begin{align}\label{lnln:eq}
Q(z; \lambda) = \ln(1-\lambda) \int\limits_{\z_1}^{\z_2} \frac{\z+z}{\z-z} d{\mathbf m}(\z)
+\ln\lambda \int\limits_{\z_2}^{\z_1} \frac{\z+z}{\z-z} d{\mathbf m}(\z). 
\end{align}
It follows from \eqref{intdm:eq}  that 
\begin{align*}
\int\limits_{\z_1}^{\z_2} \frac{\z+z}{\z-z} d{\mathbf m}(\z)
= {\mathbf m} (\z_1, \z_2) - \frac{i}{\pi} \ln\frac{1-z /\z_2 }{1-z /\z_1 },
\end{align*} 
where $2\pi {\mathbf m}(\z_1,  \z_2)$ is the length of the arc $(\z_1, \z_2)$, and 
\begin{align*}
\int\limits_{\z_2}^{\z_1} \frac{\z+z}{\z-z} d{\mathbf m}(\z)
= 1- {\mathbf m} (\z_1, \z_2) + \frac{i}{\pi} \ln\frac{1-z /\z_2}{1-z /\z_1 }.
\end{align*} 
 Consequently, the right-hand side of \eqref{lnln:eq} equals
\begin{align}
Q(z; \lambda) =\ln(1-\lambda) \,{\mathbf m}(\z_1, \z_2) 
+ &\ \ln\lambda \, (1-{\mathbf m}(\z_1, \z_2)) \nonumber \\
- &\ \frac{i}{\pi} \big(
\ln(1-\lambda) - \ln\lambda\big)\ln  \frac{1-z \z_2^{-1}}{1-z\z_1^{-1}} \nonumber \\
= &\ 
\ln\lambda 
+ 2\pi\sigma(\lambda) {\mathbf m}(\z_1, \z_2) 
-2 i\sigma(\lambda) 
\ln  \frac{1-z \z_2^{-1}}{1-z\z_1^{-1}}.
\label{eq:Qint}\end{align}
This yields the expression 
\begin{align*}
\xi(z; \lambda) = \lambda^{-\frac{1}{2}} e^{-\pi\sigma(\lambda) {\mathbf m}(\z_1, \z_2) }
(1-z / \z_1)^{-i\sigma (\lambda)} 
(1-z/\z_{2})^{i\sigma (\lambda)} 
\end{align*}
for function \e{eq:FF1}.
The coefficient $\rho(\lambda)$ is found from \eqref{simplerho:eq}.
Substituting these formulas into \eqref{simplephi:eq} we obtain \eqref{eq:RS}.
\end{proof}

Note that the functions $\varphi_{\rm s}$ are in $\mathbb H^p$ 
for any $p <2$, whereas the functions $\varphi_{\rm r}$ are in $\mathbb H^p$ for $p <1$ only.

Finally, we find an explicit expression for the function $\varphi^{\rm (ext)}_{\rm s} (z;\lambda) $. Suppose that $|z|> 1$. Then instead of \e{intdm:eq} we use
the identity \e{eq:tw},
and hence, quite similarly to \e{eq:Qint}, we find that
\[
 Q(z; \lambda) = - \ln\lambda - 2\pi\sigma(\lambda) {\mathbf m}(\z_1, \z_2) 
+ 2i\sigma(\lambda) 
\ln  \frac{1-z^{-1} \z_1 }{1- z^{-1} \z_2}.
\]
This yields an expression 
\begin{equation*}
\xi(z; \lambda) = \lambda^{\frac{1}{2}} e^{\pi\sigma(\lambda) {\mathbf m}(\z_1, \z_2) }
(1-z^{-1} \z_1)^{-i\sigma (\lambda)} 
(1- z^{-1} \z_{2})^{i\sigma (\lambda)} 
\end{equation*}
for function \e{eq:FF1}.  Substituting this expression into \e{eq:m=1} 
 and taking into account that 
\begin{align*}
e^{-\pi i {\mathbf m}(\z_2, \z_1)} = - e^{\pi i {\mathbf m}(\z_1, \z_2)},
\end{align*}
we find that
\[
\varphi^{\rm (ext)}_{\rm s} (z;\lambda) = \rho (\lambda) \lambda^{1/2}  
 e^{\pi(\sigma(\lambda) + i) {\mathbf m}(\z_1, \z_2) }  \z_1z ^{-1} 
(1- \z_1/z)^{ -1/2-i\sigma (\lambda)} 
(1-\z_{2}/z)^{ -1/2+i\sigma (\lambda)}.
\]


\section{Piecewise continuous symbols}\label{piecewise:sect}

In this section we focus on piecewise continuous symbols $\omega$. 
The results obtained here are used in \cite{SY}. For such symbols, 
Condition \ref{mult:cond} 
can be verified for appropriate spectral intervals $\Lambda$ and spectral multiplicity can be 
expressed via the counting functions of intervals of monotonicity 
and of the jumps of $\omega$. 
The resulting adaptation of Theorem \ref{specuni:thm} is stated as Theorem ~\ref{PWC:thm}. 

\subsection{Exceptional sets}  
First we explain in exact terms what we mean by a piecewise continuous symbol $\omega$. 
Below we adopt the notation $\omega'(\z) = d\omega(e^{i\theta})/d\theta$ where 
$\z = e^{i\theta}, \theta\in \mathbb R$.

\begin{cond} \label{SM:cond}
\begin{enumerate}[{\rm(i)}]

\item 
$\omega = \overline{\omega}\in L^\infty(\mathbb T)$ and $\omega$ is not a constant function, 
\item 
There exists a finite set ${\sf S} = \{\eta_k\}\subset\mathbb T$ 
such that $\omega\in C^1(\mathbb T\setminus {\sf S})$, 
\item 
The limits  
$\omega(\eta_k\pm 0) = \lim_{\varepsilon\to \pm 0} \omega(\eta_k e^{i\varepsilon})$ 
exist for all $\eta_k\in {\sf S}$.  
For every $\eta_k\in {\sf S}$, either $\omega(\eta_k+0)\neq  \omega(\eta_k-0)$ or
  $\omega(\eta_k+0)= \omega(\eta_k-0)$ but
the derivative 
$\omega'$ is not continuous at $\eta_k$.

\end{enumerate}
\end{cond}

This condition is assumed to be satisfied throughout this section.

Let ${\sf S}^{(\pm)}$ be the subset of  those $\eta_k \in {\sf S}$ for which 
\begin{equation*}
\pm (\omega(\eta_k - 0)-\omega (\eta_k + 0))>0,\q \eta_k\in {\sf S}^{(\pm)},
 \end{equation*}
 and let ${\sf S}_0$ be the set of those $\eta_k$  where 
$\omega(\eta_k - 0) = \omega(\eta_k + 0)$ but the derivative $\omega'(\z)$ is not continuous 
at the point $\eta_k$ 
%
, so 
${\sf S}$ is the disjoint union  
\begin{align*}
\sf {\sf S} = {\sf S}^{(+)}\cup {\sf S}^{(-)}\cup {\sf S}_{0}.
\end{align*}
We associate with every discontinuity $\eta_k\in {\sf S}^{(\pm)}$ the interval (the ``jump"):
\begin{align}\label{Lk:eq}
\Lambda_{k} = [\omega(\eta_k\pm 0), \omega (\eta_k \mp 0)],\q \eta_k\in {\sf S}^{(\pm)}.
\end{align}
Introduce the set 
${\sf S}_{\rm cr}\subset {\mathbb T}\setminus {\sf S}$  
of critical points where $\omega'(\z)=0$. 
The image $\Lambda_{\rm cr}=\omega ({\sf S}_{\rm cr})$ of this set consists of critical values of 
$\omega$. By Sard's theorem, the Lebesgue measure $|\Lambda_{\rm cr}|=0$.       

We introduce also the ``threshold" set $\Lambda_{\rm thr}$ which consists of all values 
$\omega(\eta_k \pm 0)$, $\eta_k\in {\sf S}$,  
and define the exceptional set 
\[
\Lambda_{\rm exc} =\Lambda_{\rm cr} \cup \Lambda_{\rm thr}.
\]
Since the set  $\Lambda_{\rm thr}$ is finite, $|\Lambda_{\rm exc}| = 0$. 

\begin{lemma} 
 The set  $\Lambda_{\rm exc}$ is closed. 
\end{lemma}
 
\begin{proof} 
Let  $\lambda_{n} \in \Lambda_{\rm exc}$ 
and $\lambda_{n} \to \lambda_0$ as 
$n\to\infty$. We may suppose that $\lambda_{n} \in \Lambda_{\rm cr}$  and consider a sequence of points $\z_{n}\in {\sf S}_{\rm cr}$   such 
that $\omega (\z_{n} ) =\lambda_{n}$. Extracting, if necessary, a subsequence, 
we assume that $\z_n\to \z_{0}$  
as $n\to\infty$. If $\z_{0}\in {\sf S}$, 
then $\omega (\z_0 \pm 0)\in  \Lambda_{\rm thr}$.  If $\z_{0}\in {\mathbb T}\setminus {\sf S}$, we use that $\omega' (\z_n) = 0$ and that $\omega' (\z)$ is continuous at the point $\z_{0}$. Thus,
 $\omega'(\z_0)   = 0$, that is, 
$\lambda_{0}= \omega (\z_0)\in \Lambda_{\rm cr}$.
\end{proof}

We need the following elementary fact about the roots of the equation $\omega(\z) = \lambda$.
 
\begin{lemma} \label{functXY}  
For all $\lambda\in(\gamma_{1}, \gamma_{2})\setminus \Lambda_{\rm exc}$, 
the set
\begin{equation*}
\mathcal N_{\lambda}= \{\z\in  \mathbb T\setminus {\sf S}: \omega(\z) = \lambda\}, 
\end{equation*} 
is finite.
\end{lemma}
 
\begin{proof} 
   Suppose,  on the contrary, 
 that there exists an infinite sequence 
 $\z_{n}\in \mathbb T\setminus {\sf S}$ 
 such that $\omega(\z_{n})=\lambda$. We may assume that 
   $\z_{n}\to \z_{0}$ as $ n\to\infty$. 
Since $\lambda\notin \Lambda_{\rm thr}$, we see that 
$\z_0\in \mathbb T\setminus {\sf S}$ whence  
$\omega(\z_0) =   \lambda$.  Thus $\omega(\z_n) =\omega(\z_0) $ for all $n$. 
As $\omega\in C^1(\mathbb T\setminus {\sf S})$, we have 
\begin{align*}
\omega'(\z_0)= \lim\limits_{n\to\infty}\frac{\omega(\z_n) - \omega(\z_0)}{\z_n-\z_0} = 0, 
\end{align*}  
 i.e.,  $\z_0\in {\sf S}_{\rm cr}$, which contradicts the assumption $\lambda\notin \Lambda_{\rm cr}$.  
 This proves the claim.
 \end{proof}
  
Along with $\mathcal N_{\lambda}$ define the sets          
\begin{align*}
\mathcal N^{(\pm)}_{\lambda}
= \{\z\in\mathcal N_{\lambda} : \mp\omega'(\lambda)>0\},
\end{align*}
and   consider the counting functions
\begin{equation}
n_{\lambda} = \#\{\mathcal N_{\lambda}\}, \q  
n^{(\pm)}_{\lambda} = \# \{\mathcal N^{(\pm)}_{\lambda}\},\q \lambda\notin \Lambda_{\rm exc}. 
\label{eq:count}
\end{equation}  
According to Lemma~\ref{functXY}  these functions take finite values.
  

\subsection{Counting functions}
Let us fix an interval $\Lambda = (\lambda_{1}, \lambda_{2}) $ such that
\begin{equation} 
\Lambda \subset (\gamma_{1}, \gamma_{2})\setminus  \Lambda_{\rm exc}
\label{eq:TH}
\end{equation} 
and consider its preimage $\omega^{-1}(\Lambda)$. According to \e{eq:TH}  we have
\begin{align*}
\omega^{-1}(\Lambda)\subset \mathbb T\setminus\big({\sf S}\cup {\sf S}_{\rm cr}\big).
\end{align*} 
The open set $\omega^{-1}(\Lambda)$ is a union of disjoint open 
arcs such that $\omega'(\z)\neq 0 $ on each 
such arc $\delta$. At the endpoints of each $\delta$ the function $\omega$ takes the values 
$\lambda_1$ and $\lambda_2$, and hence $\omega(\delta) = \Lambda$. 

Denote by $\delta^{(\pm)}_k = \delta^{(\pm)}_k(\Lambda), k = 1, 2, \dots$, 
the arcs on which $\mp\omega'(\z) >0$ so that
\begin{align*}
\omega^{-1}(\Lambda)
= \bigcup_{k=1}^{n^{(+)}} \delta^{(+)}_k  
\cup \bigcup_{k=1}^{n^{(-)}} \delta^{(-)}_k,
\end{align*}
where 
\begin{align*}
  n^{(\pm)} = n^{(\pm)} (\Lambda) = \#\{ \delta_{k}^{(\pm)}(\Lambda)\}.
\end{align*}  
As the next lemma shows, 
the number $n^{(\pm)} (\Lambda)$ is finite.
 
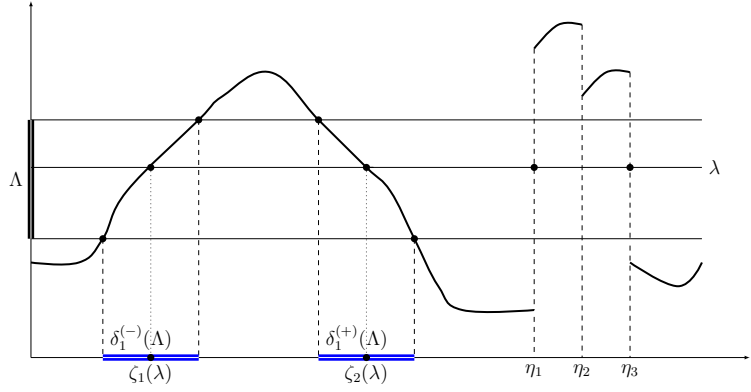
\begin{figure}
\resizebox{10cm}{!}{
\begin{tikzpicture}
\draw [very thick] plot[smooth, tension=.7] coordinates 
{(-5,-2) (-4,-2) (-3.5,-1.5) (-3,-0.5) (-1.5,1) (-1,1.5) 
(0,2) (1,1) (2,0) (2.5,-0.5) (3,-1.5) (3.5,-2.5) (4,-3) (5.5,-3)};

\draw [very thick] plot[smooth, tension=.7] coordinates {  (5.5,2.5)(6.0,3.0) (6.5,3)};
\draw [very thick] plot[smooth, tension=.7] coordinates { (6.5,1.5)(7,2)(7.5,2)};

\draw  [very thick]plot[smooth, tension=.7] coordinates {(7.5,-2) (8.5,-2.5) (9,-2)};
\draw (-5,1) node (v4) {} -- (9,1);
\draw (-5,-1.5) node (v3) {} -- (9,-1.5);
\draw[ultra thin, -latex] (-5,-4) node (v1) {} -- (10,-4) node (v2) {};
\draw [ultra thin, -latex] (-5,-4) -- (-5,3.5);
\draw [double][ultra thick](-5,-1.5) --  (-5,1)
node[pos=.5, left] {$\Lambda$};
\draw[dashed] (5.5,2.5) -- (5.5,-3) node (v11) {};
\draw[dashed] (7.5,2) -- (7.5,-1.5) node (v12) {};
\draw [dashed](-3.5,-1.5) node (v16) {} -- (-3.5,-4) node (v5) {};
\draw[dashed] (-1.5,1) -- (-1.5,-4) node (v6) {};
\draw[dashed] (1,1) -- (1,-4) node (v7) {};
\draw[dashed] (3,-1.5) -- (3,-4) node (v8) {};
\draw [double][ultra thick, blue](-3.5,-4) -- (-1.5,-4)
node[pos=.4, above, black] {$\delta^{(-)}_1(\Lambda)$};
\draw  [double][ultra thick, blue ](1,-4) -- (3,-4)
node[pos=.4, above, black] {$\delta^{(+)}_1(\Lambda)$};
\draw (-5,0)  -- (9,0)node[right] {$\lambda$};; 
\draw (2,0) node (v9) {} -- (v9) -- (2,0) node (v10) {} -- cycle;
\draw[dotted] (2, 0) node (v18) {} -- (2,-4) node (v14) {};
\draw[dashed] (v11) -- (5.5,-4) node (v15) {};
\draw[dashed] (v12) -- (7.5,-4);
 \coordinate[label=below:$\zeta_1(\lambda)$] (A) at (-2.5, -4);
\coordinate[label=below:$\zeta_2(\lambda)$] (A) at (2, -4);
\coordinate[label=below:$\eta_1$] (A) at (5.5, -4);
\coordinate[label=below:$\eta_2$] (A) at (6.5, -4);
\coordinate[label=below:$\eta_3$] (A) at (7.5, -4);
\draw [dotted](-2.5,0) node (v17) {} -- (-2.5,-4) node (v13) {};

\node at (-3.5,-1.5)[circle,fill,inner sep=1.5pt]{}; 
\node at (3,-1.5)[circle,fill,inner sep=1.5pt]{}; 
\node at (1,1)[circle,fill,inner sep=1.5pt]{}; 
\node at (-1.5,1)[circle,fill,inner sep=1.5pt]{}; 
 \node at (v17)[circle,fill,inner sep=1.5pt]{}; 
 \node at (v18)[circle,fill,inner sep=1.5pt]{}; 
  \node at (5.5, 0)[circle,fill,inner sep=1.5pt]{}; 
   \node at (7.5, 0)[circle,fill,inner sep=1.5pt]{}; 
   \node at (v13)[circle,fill,inner sep=1.5pt]{}; 
    \node at (2, -4)[circle,fill,inner sep=1.5pt]{}; 
\draw [dashed](6.5,3) -- (6.5,-4);
\end{tikzpicture}}
\caption{
Example: 
$\omega^{-1}(\Lambda) = \delta^{(+)}_1(\Lambda)\cup\delta^{(-)}_1(\Lambda)$;
$\mathcal N^{(-)}_\lambda = \{\zeta_1(\lambda)\}$, 
$\mathcal N^{(+)}_\lambda = \{\zeta_2(\lambda)\}$, 
${\sf S}^{(-)} = {\sf S}^{(-)}(\Lambda) = \{\eta_1\}$, 
${\sf S}^{(+)} = \{\eta_2, \eta_3\}$, 
${\sf S}^{(+)}(\Lambda) = \{\eta_3\}$. }
\end{figure}

\begin{lemma}
Let Condition \ref{SM:cond} hold, and 
let the counting function $n^{(\pm)}_\lambda$ be defined by formula \e{eq:count}. Then
\begin{equation}
  n^{(\pm)}_{\lambda} = n^{(\pm)} (\Lambda)   
 \label{eq:count1}\end{equation} 
for   all $\lambda\in\Lambda$. In particular, the function 
$n^{(\pm)}_{\lambda}$ is a finite constant
on the interval $\Lambda$. 
\end{lemma}

\begin{proof}
Since $\mp\omega' (\z)>0$ on each $\delta^{(\pm)}_k$, 
and $\omega(\delta^{(\pm)}_k) = \Lambda$, 
each arc $\delta^{(\pm)}_k$ contains exactly 
one point $\z\in\mathcal N^{(\pm)}_{\lambda}$ for any $\lambda\in \Lambda$. 
This implies equality \e{eq:count1}. 
By Theorem \ref{functXY} the set $\mathcal N^{(\pm)}_{\lambda}$ is finite, and hence both sides 
of \eqref{eq:count1} are finite. 
\end{proof}

Consider now the singular points. 
Define the counting function $s_\lambda^{(\pm)}$ of the intervals \eqref{Lk:eq}:
\begin{align*}
  s_\lambda^{(\pm)} = \#\{\Lambda_k: \lambda\in \Lambda_k, \eta_k\in {\sf S}^{(\pm)}\}.
\end{align*}
Since $\omega (\eta_{k}\pm 0)\not\in \Lambda$ for all $\eta_{k}\in {\sf S}$, 
each interval 
$\Lambda_k$ either contains $\Lambda$ or is 
disjoint from $\Lambda$. 
Let ${\sf S}(\Lambda)\subset {\sf S}$ be the set of those points $\eta_k$, 
for which $\Lambda\subset \Lambda_k$. 
Introduce also the notation 
${\sf S}^{(\pm)}(\Lambda) = {\sf S}(\Lambda)\cap {\sf S}^{(\pm)}$ 
and the counting function 
\begin{align}
s^{(\pm)}(\Lambda) = 
\#\{ {\sf S}^{(\pm)}(\Lambda)\}.
\label{eq:countZ}\end{align} 
It is clear that $s_\lambda^{(\pm)}$ is constant on $\Lambda$ and 
\begin{align}\label{eq:count2}
s_\lambda^{(\pm)} = s^{(\pm)}(\Lambda),\q \lambda\in \Lambda.
\end{align}
Note that ${\sf S}(\Lambda) = {\sf S}^{(+)}(\Lambda)\cup{\sf S}^{(-)}(\Lambda)$. 
All these objects are illustrated in Fig.~2 where $\z_{k}(\lambda)$, $k = 1, 2,$ 
are the solutions of the equation $\omega(\z)=\lambda$.

\subsection{ Spectral multiplicity} 
 
For piecewise continuous symbols, the number $m$ of arcs in \e{gammaf:eq} can be calculated 
in terms of the counting functions.

\begin{figure}
\resizebox{10cm}{!}{
\begin{tikzpicture}
 
\draw [very thick] plot[smooth, tension=.7] coordinates 
{(-5,-2) (-4,-2) (-3.5,-1.5) (-3,-0.5)};

\draw [very thick, -triangle 60](-3,-0.5) -- (-1.5,1);


\draw [very thick] plot[smooth, tension=.7] coordinates 
{(-1.5,1) (-1,1.5) 
(0,2) (1,1)};
\draw[very thick, -triangle 60](1,1) -- (2,0);

\draw [very thick] plot[smooth, tension=.7] coordinates 
{(2,0) (2.5,-0.5) (3,-1.5) (3.5,-2.5) (4,-3) (5.5,-3)};

\draw [very thick] plot[smooth, tension=.7] coordinates {  (5.5,2.5)(6.0,3.0) (6.5,3)};
\draw [very thick] plot[smooth, tension=.7] coordinates { (6.5,1.5)(7,2)(7.5,2)};

\draw  [very thick]plot[smooth, tension=.7] coordinates {(7.5,-2) (8.5,-2.5) (9,-2)};
\draw (-5,1) node (v4) {} -- (9,1);
\draw (-5,0)  -- (9,0)node[right] {$\lambda$};
\draw (-5,-1.5) node (v3) {} -- (9,-1.5);
\draw[ultra thin, -latex] (-5,-4) node (v1) {} -- (10,-4);
\draw [ultra thin, -latex] (-5,-4) -- (-5,3.5);
\draw [double][ultra thick](-5,-1.5) --  (-5,1)
node[pos=.5, left] {$\Lambda$};

\draw[very thick, -triangle 60] (5.5,-3) -- (5.5,0){};
\draw[very thick] (5.5,0) -- (5.5,2.5) node (v11){};


\draw[very thick](6.5,1.5)--(6.5,3);

\draw[very thick, -triangle 60](7.5,2)--(7.5,0);
\draw[very thick](7.5,0)--(7.5,-2);

\draw [double][ultra thick, blue](-5,-4) -- (-2.5,-4);

\draw [double][ultra thick, blue](2,-4) -- (5.5,-4);

\draw  [double][ultra thick, blue ](7.5,-4) -- (9,-4);

\draw (2,0) node (v9) {} -- (v9) -- (2,0) node (v10) {} -- cycle;
\draw[dashed] (2, 0) node (v18) {} -- (2,-4) node (v14) {};
\draw[dashed] (v11) -- (5.5,-4) node (v15) {};
\draw[dashed] (v12) -- (7.5,-4);

 \coordinate[label=below:$\beta_1(\lambda)$] (A) at (-2.5, -4);
\coordinate[label=below:$\alpha_2(\lambda)$] (A) at (2, -4);
\coordinate[label=below:$\beta_2(\lambda)$] (A) at (5.5, -4);
\coordinate[label=below:$\alpha_1(\lambda)$] (A) at (7.5, -4);
\draw [dashed](-2.5,0) node (v17) {} -- (-2.5,-4) node (v13) {};

 \node at (v17)[circle,fill,inner sep=1.5pt]{}; 
 \node at (v18)[circle,fill,inner sep=1.5pt]{}; 
  \node at (5.5, 0)[circle,fill,inner sep=1.5pt]{}; 
   \node at (7.5, 0)[circle,fill,inner sep=1.5pt]{}; 
  \node at (-5,0)[circle,fill,inner sep=1.5pt]{}; 
\end{tikzpicture}}
\caption{ }
\end{figure}

\begin{theorem}\label{count:thm}  
Let Condition \ref{SM:cond} hold, and let the 
interval $\Lambda= (\lambda_1, \lambda_2)$ satisfy condition \e{eq:TH}. Then
for each $\lambda\in\Lambda$, the set $\Gamma(\lambda)$ defined by \e{eq:mult} is the union \e{gammaf:eq}
 of finitely many  open arcs 
$(\alpha_{j} (\lambda), \beta_{j}(\lambda))
\subset \mathbb T \setminus {\sf S} (\Lambda)$
such that their closures are disjoint. The 
number $m= m(\Lambda) $ of these arcs does not depend on $\lambda \in\Lambda$ and, for both signs $``\pm"$, 
\begin{align}\label{count:eq}
m(\Lambda) = n^{(\pm)}(\Lambda) + s^{(\pm)}(\Lambda)
\end{align} 
where $n^{(\pm)}(\Lambda) $ and $ s^{(\pm)}(\Lambda)$ 
are defined by \e{eq:count1} and \e{eq:countZ}, respectively.
\end{theorem}

\begin{proof} 
Join the points $(\eta_k, \omega(\eta_k-0))\in  \mathbb T\times \mathbb R$ 
and $(\eta_k, \omega(\eta_k+0))\in  \mathbb T\times \mathbb R$ for all 
$\eta_k\in \mathbb {\sf S}^{(+)}\cup{\sf S}^{(-)}$ with straight segments. These 
segments, together with the graph 
$\{(\z, \omega(\z)): \z\in \mathbb T\setminus {\sf S}\}\subset \mathbb T\times \mathbb R$, 
form a closed non-self-intersecting continuous curve $\mathcal C$ 
on the cylinder $\mathbb T\times \mathbb R$. 
 
Since $\lambda\in \Lambda$, 
the curve $\mathcal C$  crosses the straight line 
$\{(\z, \lambda): \z\in\mathbb T\}$ 
at the points of the sets $\mathcal N_\lambda$ and ${\sf S}(\Lambda)$ only. As $\mathcal C$ is continuous, 
the points of intersection in the downward direction, denoted 
by $\alpha_j(\lambda)\in\mathbb T$, alternate with 
those in the upward direction, denoted by $\beta_j(\lambda)\in \mathbb T$, 
see Fig.~3 for illustration. Therefore the quantities of the downward and upward points are the same:
\begin{align*}
m_\lambda := \#\{\alpha_j(\lambda), j=1, 2, \dots\} = \#\{\beta_j(\lambda), j = 1, 2, \dots\}. 
\end{align*} 
It follows that up to a set of measure zero, the set 
$\Gamma(\lambda)$ consists of $m_\lambda$ disjoint open intervals  
$(\alpha_j(\lambda), \beta_j(\lambda))$ 
with disjoint closures. 
It is easy to see that,  for all $\lambda\in \Lambda$,
 \[
 \{\alpha_{1}(\lambda), \alpha_{2}(\lambda), \ldots\}=\mathcal N^{(+)}_\lambda\cup {\sf S}^{(+)}(\Lambda)\q  {\rm and}\q
  \{\beta_{1}(\lambda), \beta_{2}(\lambda), \ldots\}=\mathcal N^{(-)}_\lambda\cup {\sf S}^{(-)}(\Lambda).
 \]
 These sets are finite and, by  \eqref{eq:count1} and \eqref{eq:count2}, they consist  of  $n^{(+)}(\Lambda) + s^{(+)}(\Lambda)$ and $n^{(-)}(\Lambda) + s^{(-)}(\Lambda)$ points, respectively, whence
  \[
m_{\lambda}=  n^{(\pm)}(\Lambda) + s^{(\pm)}(\Lambda)
  \]
for both signs.  
\end{proof}

In the example in Fig.~2 we have
$n^{(\pm)}(\Lambda) = s^{(\pm)}(\Lambda)=1$, so that $m=2$.

Putting together Theorems~\ref{specuni:thm} and \ref{count:thm}, we obtain our final result.

\begin{theorem}\label{PWC:thm}
Suppose that $\omega$ satisfies Condition \ref{SM:cond}, and that 
an interval $\Lambda$ satisfies \eqref{eq:TH}.  Let the number  $m$ be defined in \eqref{count:eq}. 
Then the 
spectral representation of  
the operator $T$ restricted to the 
subspace $E(\Lambda){\mathbb H}^{2}$ 
is realized on the space 
$L^{2}(\Lambda; {\mathbb C}^m)$. 
In other words, the spectral multiplicity 
of the operator $T$ on the interval $\Lambda$ is finite, and it coincides with the 
number $m$.
\end{theorem}

  This result is crucial for our construction of scattering theory for piecewise continuous 
symbols in \cite{SY}.

 \end{document}